\documentclass{article}
\usepackage{amsmath}
\usepackage{amssymb}
\usepackage{amsfonts}
\usepackage{amsthm}
\usepackage{url}
\usepackage{graphicx}
\usepackage{subfigure}
\usepackage{algorithm}
\usepackage[noend]{algpseudocode}

\usepackage{tikz}
\usepackage{comment}

\newcommand{\ProbB}[1]{\ensuremath{\mathbb{P} \left(#1 \right) }}
\newcommand{\EE}[1]{\ensuremath{\mathbb{E} \left[#1 \right]}}
\newcommand{\E}{\ensuremath{\mathbb{E}}}
\newcommand{\V}{\ensuremath{\mathrm{Var}}}

\newcommand{\Prob}{\ensuremath{\mathbb{P}}}
\newcommand{\nc}{\newcommand}
\nc{\Nset}{{\mathbb{N}}} \nc{\Rset}{{\mathbb{R}}}
\nc{\R}{{\mathbb{R}}} \nc{\N}{{\mathbb{N}}}
\nc{\Zset}{{\mathbb{Z}}}
\newcommand{\Do}{\ensuremath{\mathrm{D}([0,1])}}
\newcommand{\Doab}{\ensuremath{\mathrm{D}([a,b])}}

\newcommand{\Co}{\ensuremath{\mathrm{C}(\Sigma^\infty)}}	

\newcommand{\Law}{\ensuremath{\mathcal{L}}}

\begin{document}
\newtheorem{theorem}{Theorem}[section]
\newtheorem{korollar}[theorem]{Corollary}
\newtheorem{lemma}[theorem]{Lemma}
\newtheorem{proposition}[theorem]{Proposition}
\renewcommand{\labelenumi}{\roman{enumi}) }

\title{\bf Process convergence for the complexity of \\
Radix Selection on Markov sources}
\author{Kevin Leckey \\School of Mathematical Sciences\\
Monash University \\ VIC 3800, Melbourne\\
Australia \and
Ralph Neininger\\
Institute for Mathematics\\
Goethe University Frankfurt\\
60054 Frankfurt am Main\\
Germany
\and
Henning Sulzbach\thanks{Corresponding author. Email: henning.sulzbach@gmail.com. Present address: School of Mathematics, University of Birmingham, Birmingham B15 2TT, Great
Britain}\\
School of Computer Science\\
McGill University\\ H3A 2K6, Montreal\\
Canada
}

\date{\today}
\maketitle

\begin{abstract}
A fundamental algorithm for selecting ranks from a finite subset of an ordered set is Radix Selection. This algorithm requires the data to be given as strings of symbols over an ordered alphabet, e.g., binary expansions of real numbers. Its complexity is measured by the number of symbols that have to be read. In this paper the model of independent data identically  generated from a Markov chain is considered.

The complexity is studied as a stochastic process indexed by the set of infinite strings over the given alphabet. The orders of mean and variance of the complexity and, after normalization, a limit theorem with a centered Gaussian process as limit are derived. This implies an analysis for two standard models for the ranks: uniformly chosen ranks, also called grand averages, and the worst case rank complexities which are of interest in computer science.

For uniform data and the asymmetric Bernoulli model (i.e.~memoryless sources), we also find weak convergence for the normalized process of complexities when indexed by the ranks while for more general Markov sources these processes are not tight under the standard normalizations.
\end{abstract}


\section{Introduction} \label{sec:intro}
In the probabilistic analysis of algorithms the complexity of fundamental algorithms is studied under models of random input. This allows to describe the typical behavior of an algorithm and is often more meaningful than the worst case  complexity  classically considered in computer science. In this paper we study the algorithm Radix Selection on independent strings generated by a Markov source.

Radix Selection  selects an order statistic from a
set of data in $[0,1]$ as follows. First, an integer $b\ge 2$ is fixed and the  unit interval is
decomposed into the intervals, also called {\em buckets}, $[0,1/b), [1/b,2/b), \ldots, [(b-2)/b,(b-1)/b)$ and $[(b-1)/b,1]$. The data are assigned to these buckets according to their values. Note that this just corresponds to grouping the data according to the first symbols of their $b$-ary expansions. If the bucket containing the datum with rank to be selected contains further data, the algorithm is recursively applied by decomposing this bucket equidistantly using the integer $b$. The algorithm stops once the bucket containing the sought rank  contains no other data. 
Assigning a datum to a bucket is called a {\em bucket operation}, and the algorithm's complexity is measured by the total number of bucket operations required. An illustration of this procedure is given in Figure \ref{pic_radsel}. An algorithmic formulation of the routine is given in the appendix.
\begin{figure}
\begin{center}
\scalebox{0.7}{
\begin{tikzpicture}[edge from parent/.style={draw,-latex}]
\tikzstyle{level 0}=[level distance=1cm]
\tikzstyle{level 1}=[sibling distance=50mm,  level distance=2.5cm]
\tikzstyle{level 2}=[sibling distance=27.5mm,  level distance=2cm]
\tikzstyle{level 3}=[sibling distance=25mm, level distance=2cm]
\tikzstyle{level 4}=[sibling distance=22.5mm, level distance=1.5cm]
\node[fill=green!15,draw,text width=5em, text centered]{$1101\ldots$ \\ $0001\ldots$ \\ $0110\ldots$ \\ $0000\ldots$ \\ $1111\ldots$ \\$1110\ldots$}
  child{node[fill=green!15,draw,text width=5em, text centered]{$0001\ldots$ \\ $0110\ldots$ \\ $0000\ldots$ \\}
    child{node[fill=green!15,draw,text width=5em, text centered]{$0001\ldots$ \\ $0000\ldots$}
      child{node[fill=green!15,draw,text width=5em, text centered]{$0001\ldots$ \\$0000\ldots$}
	child{node[fill=red!15,draw,text width=5em, text centered]{$0000\ldots$}
	}
	child{node[fill=green!15,draw,text width=5em, text centered]{$0001\ldots$}
	}
      }
      child{node[fill=red!15,draw,text width=5em, text centered]{\;}
      }
    }
    child{node[fill=red!15,draw,text width=5em, text centered]{$0110\ldots$}
    }
  }
  child{node[fill=red!15,draw,text width=5em, text centered]{$1101\ldots$ \\ $1111\ldots$ \\$1110\ldots$}
  }
;
\end{tikzpicture}
}
\end{center}
\caption{A schematic representation of Radix Selection with $b=2$ buckets
 searching for the element of rank $2$ in a list of $6$ elements given in their binary expansions. Arrows indicate the splitting into buckets, the green color indicates  buckets containing the element of rank $2$. The total number of bucket operations is $6+3+2+2=13$. }
\label{pic_radsel}
\end{figure}
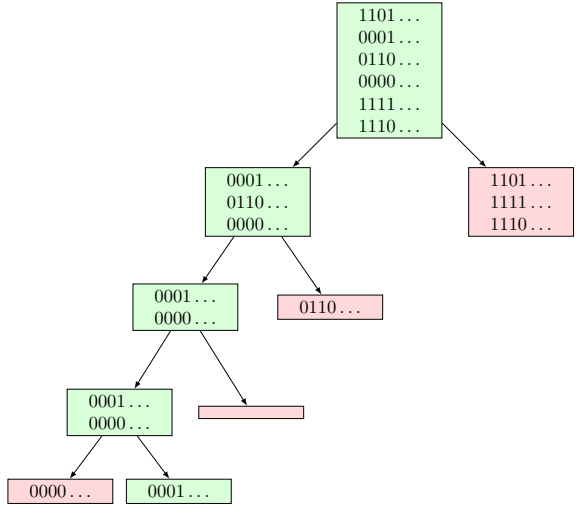

Radix Selection is especially suitable when data are stored as  expansions in base (radix) $b$, the case $b=2$ being the most common on the level of machine data. For such expansions a bucket operation breaks down to access a digit (or bit).

\medskip
{\bf The Markov source model and complexity.} We study the complexity of Radix Selection in  the probabilistic setting that $n$ data are modeled independently with $b$-ary expansions generated from a homogeneous Markov chain on the alphabet $\Sigma=\{0,\ldots,b-1\}$ with a fixed integer $b\ge 2$.  The Markov chain is
characterized by its initial distribution $\mu=\sum_{r=0}^{b-1}\mu_r \delta_r$ where $\mu_r\in [0,1]$ and $\sum_{r=0}^{b-1}\mu_r=1$,
and its transition matrix $(p_{ij})_{i,j \in \Sigma}$. Here, $\delta_x$ is the Dirac
 measure in $x\in \R$. We always assume that $0 < \mu_i, p_{ij} < 1$ for all $i, j \in \Sigma$.\footnote{This assumption greatly simplifies proofs. All theorems in Section \ref{sec:intro} remain true for arbitrary $\mu$ if the chain is irreducible and $0 < p_{ij} < 1$ for some $i,j \in \Sigma$.}

Let $\Sigma^\infty$  denote the set of infinite strings over the alphabet $\Sigma$.
For two strings $v = (v_i)_{i \geq 1}, w = (w_i)_{i \geq 1} \in \Sigma^\infty$,  we denote the length of the longest common prefix, the so-called {\em string coincidence}, by
\begin{align} \label{def:coin} j(v,w) = \max\{ i \in \N : (v_1, \ldots, v_i) =  (w_1, \ldots, w_i) \}. \end{align}
 We write $v < w$ if $v_{j(v,w)+1} < w_{j(v,w)+1}$.
Let $S_1, S_2, \ldots$ be a sequence of independent  strings generated by our Markov source. For $v \in \Sigma^\infty$ we set
$$\Lambda_{n,k}(v) = \# \{1 \leq i \leq n : j(v, S_i) \geq k \}, \quad k \geq 0,$$
and
\begin{align}\label{def:z}Z_n(v) = \sum_{k \geq 0} \Lambda_{n,k}(v) \mathbf 1_{\Lambda_{n,k}(v) > 1}.\end{align}
For the complexity of Radix Selection, i.e., the number of bucket operations $Y_n(\ell)$ necessary to retrieve the element of rank $1\leq \ell \leq n$ in the set $\{S_1, \ldots, S_n\}$, we obtain  
\begin{align}\label{def:y_n}Y_n(\ell) = Z_n(S_{(\ell)}), \quad 1 \leq \ell \leq n.\end{align}
Here, and subsequently, we write $S_{(1)} < \cdots < S_{(n)}$ for the order statistics of the strings.

We call our Markov source {\em memoryless} if all symbols within  data are  independent and identically distributed over $\Sigma$. Equivalently, for all $r \in \Sigma$, we have $\mu_r = p_{0r} = \cdots = p_{(b-1)r} =: p_r$.
For $b=2$, the memoryless case is also called the {\em Bernoulli model}. The case $\mu_i = p_{ij}=1/b$ for all $i,j\in\Sigma$  is the case of a memoryless source where all symbols are uniformly distributed over $\Sigma$.  We call this the {\em uniform model}.

\medskip
{\bf Scope of analysis.} We study the complexity to select ranks using three models for the ranks. First, all possible ranks are considered simultaneously. Hence, we consider the stochastic process of the complexities indexed by the ranks $1,\ldots,n$. We choose a scaling in time and space which asymptotically gives access to the complexity to select quantiles from the data, i.e., ranks of the size (roughly) $t n$ with $t\in[0,1]$. We call this model for the ranks the {\em quantile-model}. Second, we consider the complexity of a random rank uniformly distributed over $\{1,\ldots, n\}$ and independent from the data. This is the model proposed and studied (in the uniform model) in Mahmoud et al. \cite{mafljare00}. The complexities of all ranks are averaged in this model and, in accordance with the literature, we call it the model of {\em grand averages}. Third, we study the worst rank complexity. Here, the data are still random and the worst case is taken over the possible ranks $\{1,\ldots, n\}$. We call this {\em worst case rank}.

\medskip
{\bf Function spaces.} In the quantile-model we formulate functional limit theorems in two different spaces. First, we endow $\Sigma^\infty$ with the topology $\mathcal T_\infty$ where $v(n) \to v$ if and only if $j(v(n), v) \to \infty$. For any $a \in (0,1)$, the ultrametric $d_a(v,w) = a^{j(v,w)}, v, w \in \Sigma^\infty$ generates $\mathcal T_\infty$. It is easy to see that  $(\Sigma^\infty,\mathcal T_\infty)$ is a compact space. Let $\Co$  denote the space of continuous functions $f : \Sigma^\infty \to \R$ endowed with the supremum norm $\|f \| = \sup \{  |f(v)| : v \in \Sigma^\infty\}$. As $\Sigma^\infty$ is compact, $\Co$ is a separable Banach space. 

Second, we use the  space of real-valued c{\`a}dl{\`a}g functions $\Do$ on the unit interval.  A function $f: [0,1] \to \R$ is c{\`a}dl{\`a}g,  if, for all $t \in [0,1)$,  \begin{align*}
f(t) = f(t+) := \lim_{s \downarrow t} f(s), \end{align*} and the following limit exists for all $t \in (0,1]$:
\begin{align*}
 f(t-) := \lim_{s \uparrow t} f(s). \end{align*}
We define $\|f \|= \sup \{|f(t)| : t \in [0,1]\}$, and note that $\|f\| < \infty$ for all $f \in \Do$. The standard topology  on $\Do$ is Skorokhod's $J_1$-topology turning $\Do$ into a Polish space. A sequence of c{\`a}dl{\`a}g functions $f_n, n \geq 1$ converges to $f \in \Do$ if and only if there exist strictly increasing continuous bijections $\lambda_n, n\geq 1$ on $[0,1]$ such that, both
$\lambda_n \to \text{id}$ and $f_n\circ \lambda_n \to f$ uniformly on $[0,1]$. The space $\Doab$ is defined analogously based on the closed interval $[a,b]$ with
$- \infty < a < b < \infty$.
For more details on c{\`a}dl{\`a}g functions, we refer to Billingsley's book \cite[Chapter 3]{Billingsley1999}.

\medskip
{\bf Fundamental quantities.}  Let $\Sigma^* = \bigcup_{n \geq 0} \Sigma^n$ with the convention $\Sigma^0=\{\emptyset\}$. Further, let $\Sigma^\infty_0$ ($\Sigma^\infty_{b-1}$, respectively) be the set of infinite strings with a finite and non-zero number of entries different from $0$ ($b-1$, respectively). Note that $\Sigma^\infty_0$ and
$\Sigma^\infty_{b-1}$ are
countably infinite.

For $v = v_1\ldots v_k \in \Sigma^*$, let $\pi(v)$ denote the probability that the Markov chain starts with prefix $v$:
$$\pi(v) = \mu_{v_1} p_{v_1v_2} \cdots p_{v_{k-1}v_k}, \quad k \geq 1, \quad \text{and} \quad\pi(\emptyset) := 1.$$
$\pi(v)$ is called the {\em fundamental probability} associated with $v$.
For $v \in \Sigma^*$ and $w \in \Sigma^* \cup \Sigma^\infty$, we write $v \preceq w$ if $w$ starts with prefix $v$, that is, $w = v w'$ for some $w' \in \Sigma^* \cup \Sigma^\infty$.
Upon considering a finite length vector as infinite string by attaching an infinite number of zeros, we shall extend the definition of $j(v,w)$ in \eqref{def:coin}
and the relation $<$ to $v,w \in \Sigma^* \cup \Sigma^\infty$. We further write $v \leq w$ if $v < w$ or $v = w$.
The following three functions  play a major role throughout the paper:
\begin{align*}
m(w) & := \sum_{v \preceq w}\pi(v), \quad w \in \Sigma^\infty, \\
F(w) & := \Prob \left\{S_1 \leq w\right\} =   \lim_{n \to \infty} \sum_{v \leq w, v \in \Sigma^n}\pi(v), \quad w \in \Sigma^\infty, \\
h(t) & := \sup \{v \in \Sigma^\infty : F(v) \leq t \}, \quad t \in [0,1].
\end{align*}
Note that $m, F \in \Co$ and $h([0,1]) = \Sigma^\infty \setminus \Sigma_{b-1}^\infty$. The function $h$ is continuous at all points $t\notin F(\Sigma^\infty_0)$. For $t \in F(\Sigma^\infty_0)$, we have $h(t) = \lim_{s \downarrow t} h(s)$, and the limit $h(t-) := \lim_{s \uparrow t} h(s)$ exists.
For more details on $m, F$ and  $h$ and explicit expressions in the uniform model and for certain memoryless sources covering the Bernoulli model (with $b=2$) we refer to Section \ref{sec:prel} and Proposition \ref{prop:elem2}.

 Finally, note that these definitions extend straightforwardly to a general probabilistic source, that is, a probability distribution on $\Sigma^\infty$. Here, we only require that two independent strings are almost surely distinct, that is, $\ProbB{S_1 = v} = 0$ for all $v \in \Sigma^\infty$.

\medskip
{\bf Main results.} The main results of this work concern the asymptotic orders of mean and variance as well as limit laws for the complexity of Radix Selection for our Markov source model for all three models of ranks.

\medskip {\em Quantile-model.}  We start with the first order behaviours of the processes $Z_n, n \geq 1$ and $Y_n, n \geq 1$ defined in (\ref{def:z}) and (\ref{def:y_n}).

\begin{theorem} \label{prop:mean}
Consider Radix Selection using $b\geq 2$ buckets under the  Markov source model.
\begin{enumerate}
\item For all $v \in \Sigma^\infty$ and $\Sigma^\infty$-valued sequences $v(n), n \geq 1$, with $j(v(n), v) \to \infty$, we have, almost surely and with respect to all moments,

\begin{align*}  \frac{Z_n(v(n))}{n} \to m(v). \end{align*}
\item  For $k = k(n) \in \{1, \ldots, n\}$ with $k/n \to t \in [0,1] \setminus F(\Sigma_0^\infty)$,
we have, almost surely and with respect to all moments,
\begin{align*} \frac{Y_n(k)}{n} \to m \circ h(t). \end{align*}
\end{enumerate}
\end{theorem}
 The first order behaviour of $Y_n(k)$ for $k/n \to t \in F(\Sigma^\infty_0)$ is studied in Proposition
\ref{boundarycase}. Both statements in the  proposition remain valid for a general probabilistic source under weak conditions.
See Corollary \ref{gensource}.

\begin{figure}[tt]
\subfigure[$p_{00}$=0.2, $p_{10}$=0.5\label{subfig_a}]{
\mbox{
 \includegraphics[trim=0cm 2.5cm 0cm 1cm, clip=true,width=0.49\textwidth]{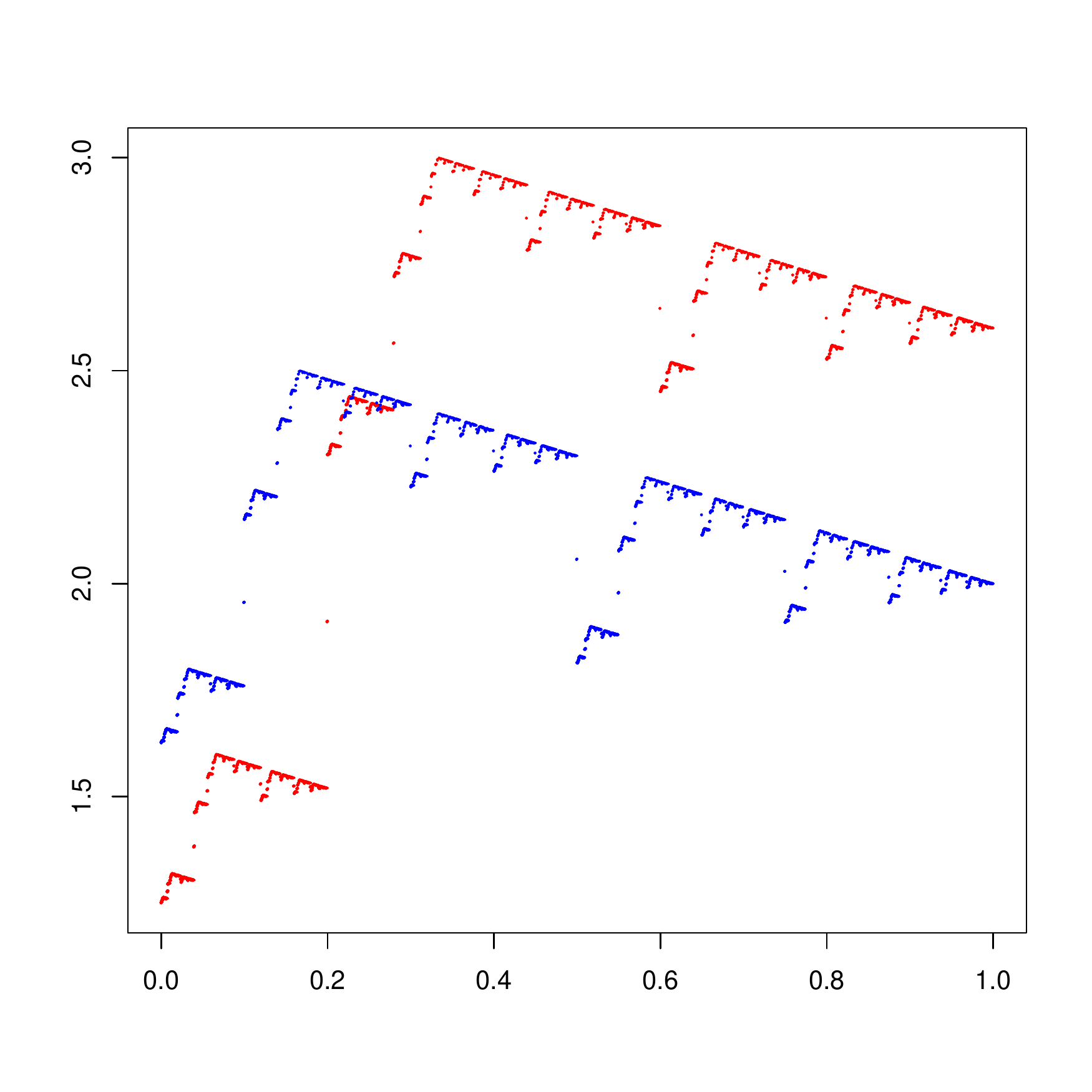}
}
}
\subfigure[$p_{00}$=0.6, $p_{10}$=0.4\label{subfig_b}]{
\mbox{
\includegraphics[trim=0cm 2.5cm 0cm 1cm, clip=true,width=0.49\textwidth]{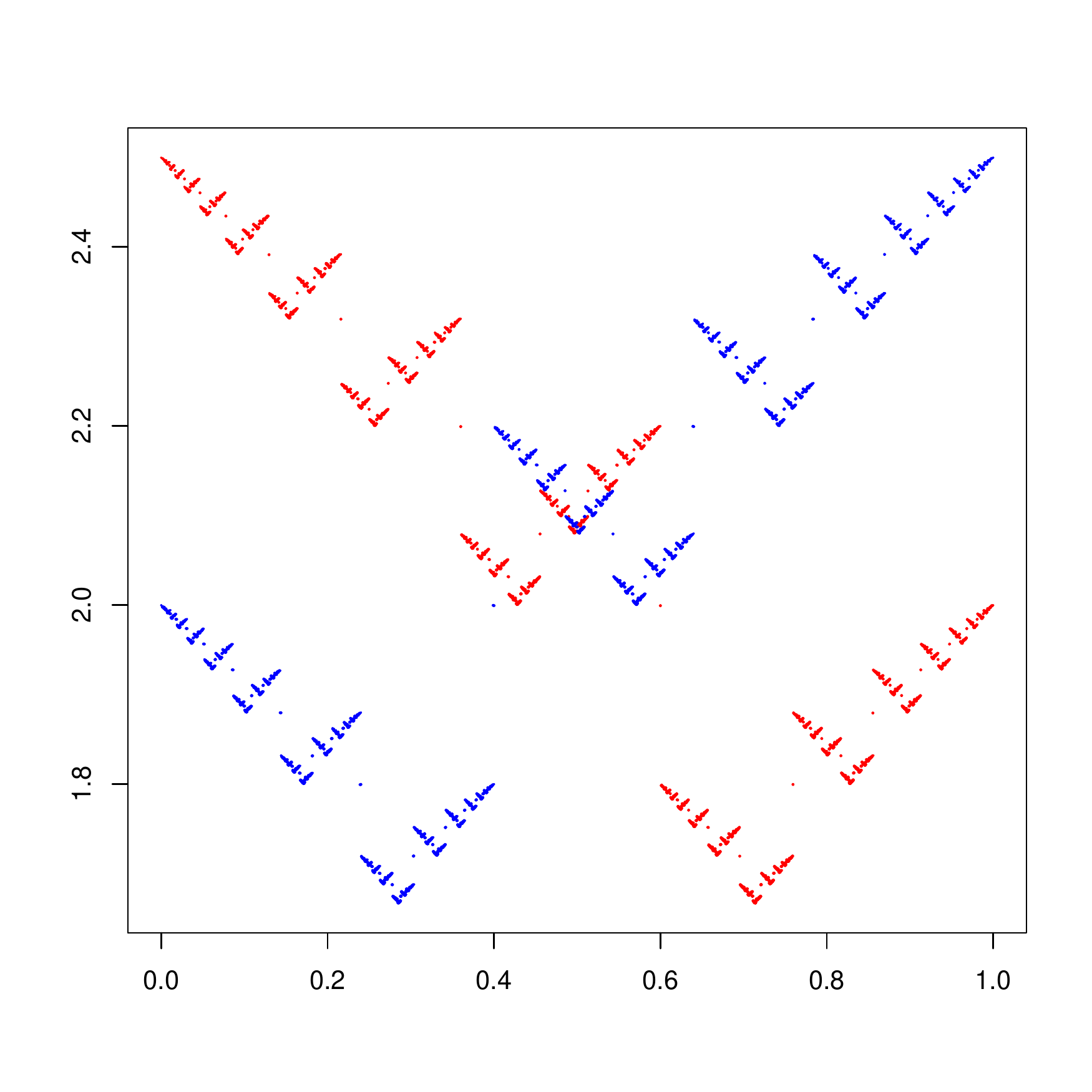}
}
}
\subfigure[$p_{00}$=0.6, $p_{10}$=0.7\label{subfig_c}]{
 \includegraphics[trim=0cm 2.5cm 0cm 1cm, clip=true,width=0.49\textwidth]{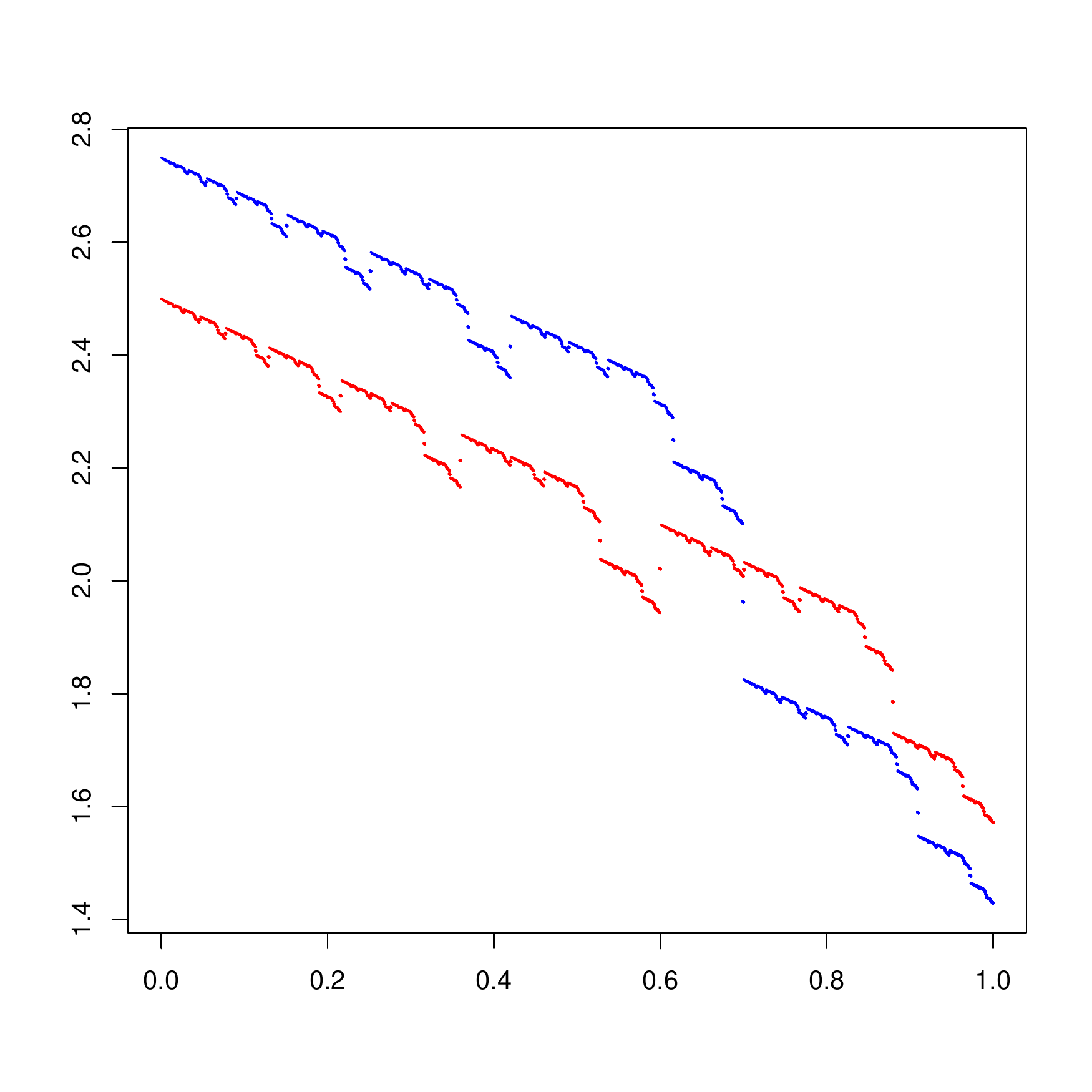}
}
\subfigure[$p_{00}$=0.3, $p_{10}$=0.7\label{subfig_d}]{
\hspace{0.25cm}
\includegraphics[trim=0cm 2.5cm 0cm 1cm, clip=true,width=0.49\textwidth]{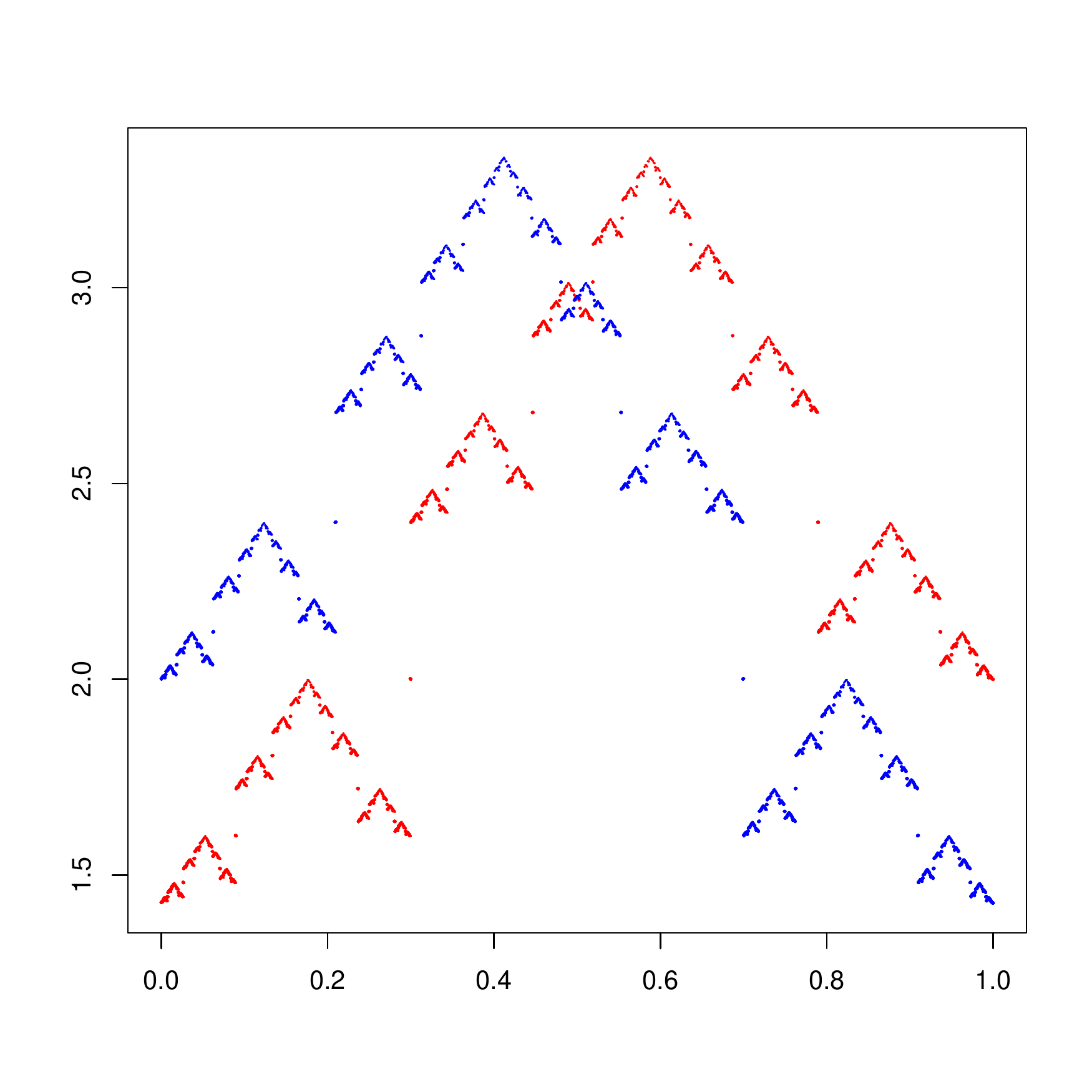}
}
\caption{Plots of $m_0 \circ h_0$ (red) and $m_1 \circ h_1$ (blue) for different Markov sources ($b=2)$.
}
\label{fig_m_0_m1}
\end{figure}

For the process $Z(v)$, we can show a functional limit theorem.

\begin{theorem} \label{thm_simple2}
 Let $b\geq 2$ and consider Radix Selection with a Markov source.
 In distribution, in $\Co$,
 \begin{align*} 
  \left( \frac{Z_n(v)- m(v) n}{\sqrt n} \right)_{v \in \Sigma^\infty}
   \to H.
\end{align*}
Here,  $H$ is a centered $\Co$-valued Gaussian process with covariance function
$$\E[H(v)H(w)] = \emph{Cov}(j(v, S_1), j(w,S_1)), \quad v,w \in \Sigma^\infty.$$
In the uniform model, we have
\begin{align*} 
\E[H(v) H(w)]=\frac b {(b-1)^2}- \frac{b+1}{(b-1)^2} b^{-j(v,w)},\quad v,w\in \Sigma^\infty.
\end{align*}
\end{theorem}
We refer the reader to Proposition \ref{prop:weaker} $iii)$ for a more explicit formula for the covariance function of $H$ for a general Markov source.

In a specific instance of a memoryless source covering both the uniform model and the Bernoulli model (with $b=2$), we can state similar results for the process $Y_n$. Here and subsequently, we always set $Y_n(n+1) :=Y_n(n)$.

\begin{theorem} \label{thm_simple}
Consider Radix Selection using $b\geq 2$ buckets under a memoryless source with \begin{align} \label{goodcase} \mu_r = p_r = \frac{1-\beta}{1-\beta^b} \beta^r, \quad r \in \Sigma, \beta > 0, \end{align}
with the convention that $\mu_r = p_r = 1 / b$ for all $r \in \Sigma$ if $\beta = 1$.
With \begin{align} \label{def:ga} \gamma = \frac{1-\beta^b}{\beta - \beta^b}, \quad \alpha = (\beta -1)\gamma, \end{align}
and the convention that $\gamma = b / (b-1)$ for $\beta = 1$, we have $m\circ h(t) = \alpha t + \gamma, t \in [0,1]$, and, as $n\to\infty$, in distribution, in $\Do$:
\begin{align*} 
  \left( \frac{Y_n(\lfloor t n \rfloor  +1) - \left(\alpha t  + \gamma \right)n}{\sqrt n} \right)_{t \in [0,1]}
   \to G,
\end{align*}
where $G$ is a centered Gaussian process with covariance function given, for $b=2$, by
\begin{align}\label{CV_G_pf}
\EE{G(s)G(t)}=-\pi(w) +\sum_{k=1}^{j(h(t),h(s))}\frac{\pi(w_1\ldots w_k)}{1-p_{w_k}}, \quad w = h(t)_1 \ldots h(t)_{j(h(t), h(s))} .
\end{align}
 Display \eqref{exp_var} gives an explicit formula for the variance for any $b \geq 2$. Almost surely, $G$ is continuous at all points $t \in F(\Sigma_0^\infty)$ and discontinuous at all points $t \notin F(\Sigma_0^\infty)$. We have $G = H \circ h$ with $H$ as in Theorem \ref{thm_simple2} in the respective model if and only if $\beta = 1$.
\end{theorem}

For an arbitrary initial distribution $\mu$ and transition probabilities as in \eqref{goodcase}, process convergence remains true on suitable subintervals of $[0,1]$. However, for a general Markov source with transition probabilities not covered by \eqref{goodcase},   process convergence for $Y_n(\lfloor tn \rfloor +1), n \geq 1$ under a scaling as in the theorem does not hold on any interval $[a,b] \subseteq [0,1]$. See Corollary \ref{cor:notight}.

\medskip {\em Grand averages.}
In the uniform model Mahmoud et al.~\cite{mafljare00} showed the following results using techniques from analytics combinatorics. Let $I_n$ be uniformly distributed on $\{1, \ldots, n\}$ and independent of $S_1, \ldots, S_n$. Then,
$$\EE{Y_n(I_n)} = \frac{b}{b-1} n + O(\log n), \quad \text{Var}(Y_n(I_n)) = \frac{b}{(b-1)^2} n + O(\log^2 n),$$
and, with convergence in distribution,
$$\frac{Y_n(I_n) - \frac{b}{b-1} n }{\sqrt{bn} / (b-1)} \to \mathcal N,$$
where $\mathcal N$ has the standard normal distribution.
Our Theorems \ref{prop:mean} -- \ref{thm_simple} cover first order expansions for mean and variance as well as the central limit theorem. 

For general Markov sources other than uniform, we find that  the limit distribution is no longer normal, and the complexity is less concentrated. Setting $k = \lfloor tn \rfloor$ and  integrating over $t \in [0,1]$, the law of large numbers  in Theorem \ref{prop:mean} $ii)$ implies the following result.

\begin{korollar} \label{thm:maingrand}
Let $b\geq 2$ and consider Radix Selection with a general Markov source. Choose $I_n$ uniformly on $\{1, \ldots, n\}$ and independently of all remaining quantities. Then, in distribution and with convergence of all moments, as $n \to \infty$,
\begin{align*}
\frac{Y_n(I_n)}{n} \to m(S_1) \stackrel{d}{=} m \circ h(\xi), \end{align*}
where $\xi$ is uniformly distributed on $[0,1]$.
\end{korollar}
We give further information on the limiting distribution in Proposition \ref{propZ}.

\medskip {\em Worst case rank.} The worst case behaviour of the algorithm is determined by the fluctuations at vectors $v \in \Sigma_{\text{max}}$ where
$$m_{\text{max}} := \max \{ m(v) : v \in \Sigma^\infty \}, \quad \Sigma_{\text{max}} := \{ v \in \Sigma^\infty : m(v) = m_{\text{max}} \}.$$
Note that $\max_{1 \leq \ell \leq n} Y_n(\ell) = \max_{v \in \Sigma^\infty} Z_n(v)$. In the memoryless case, $$\Sigma_{\text{max}} = \{ v \in \Sigma^\infty: p_{v_i} = \max_r p_r \text{ for all } i \geq 1\},$$
and $m_{\text{max}} = 1 /(1-\max_r p_r)$. In particular,
\begin{itemize}
\item in the situation of Theorem \ref{thm_simple} with $\beta \neq 1$, either $m_{\text{max}} = 1 / (1-p_0)$ and $\Sigma_{\text{max}} = \{00\ldots\}$ or
$m_{\text{max}} = 1 / (1-p_{b-1})$ and $\Sigma_{\text{max}} = \{(b-1)(b-1)\ldots\}$ depending on whether $\beta < 1$ or $\beta > 1$,
\item in the uniform model, $m_{\text{max}} = b / (b-1)$ and $\Sigma_{\text{max}} = \Sigma^\infty$.
\end{itemize}
\begin{theorem} \label{thm:mainsup}
Let $b \geq 2$ and consider Radix Selection with a Markov source. In distribution and with convergence of all moments, as $n \to \infty$,
$$\frac{ \max_{v \in \Sigma^\infty} Z_n(v) - m_{\emph{max}} n}{\sqrt{n}} \to \sup_{v \in \Sigma_{\emph{max}}} H(v).$$

\end{theorem}
The distribution of the supremum of $H$ is studied in Proposition \ref{suptails}.
The set $\Sigma_{\text{max}}$ typically contains exactly one element. $\Sigma_{\text{max}} = \Sigma^\infty$ holds only in the uniform model. In general, $\Sigma_{\text{max}}$ can be finite, countably infinite or uncountable.  See the examples at the end of  Section \ref{sec_worst_case_rank}.

\medskip
{\bf Related work.}
A general reference on bucket algorithms is Devroye \cite{de86}. A large body of probabilistic analysis of digital structures is based on methods from analytic combinatorics, see Flajolet and Sedgewick \cite{flse09}, Knuth \cite{kn98} and Szpankowski \cite{sz01}. For an approach based on renewal theory see Janson \cite{ja12} and the references given therein. Our Markov source model is a special case of the model of dynamical sources, see Cl{\'e}ment, Flajolet and Vall{\'e}e \cite{clflva01} as well as \cite{HuVa14,CeVa15}. A related important model is the density model studied in Devroye \cite{de92}. A fundamental related comparison-based selection algorithm is Quickselect (or FIND).  Below, we give a description of the routine and a detailed comparison with Radix Selection.

Let us also discuss the distributional functional limit theorems such as Theorem \ref{thm_simple} in a wider context.  In recent years, starting with Gr{\"u}bel and R{\"o}sler's seminal work \cite{GR}
on the complexity of  Quickselect,  the probabilistic analysis of several algorithms and data
structures has led to a number of functional limit theorems involving interesting non-standard limiting distributions. Apart from \cite{GR}, we can name  the study of
the profile of random search trees \cite{CKMR, DJN}, the silhouette of binary search trees \cite{grsil}, partial match queries in multidimensional search trees \cite{BNS}, Quickselect variants with increasing median selection
\cite{sunedr14} and the Quicksort process \cite{raro}. Limiting processes can be divided into four groups: processes with smooth sample paths \cite{CKMR}, \cite{DJN}, processes with continuous but non-differentiable paths \cite{grsil}, \cite{BNS}, processes with discontinuities on a random dense subset of the parameter space \cite{GR}, \cite{raro} and Gaussian limits with discontinuities on a deterministic dense subset of the parameter space \cite{sunedr14}. The limiting processes arising in this work fall in the last group. As in \cite{sunedr14}, these processes have continuous sample paths with respect to a suitable non-Euclidean topology.

\medskip
{\bf Radix Sorting and tries.}
The Radix Sorting algorithm  starts by assigning all data to the buckets as for Radix Selection. Then  it recurses on all buckets containing more than one datum. This leads to a sorting algorithm whose complexity  is measured by the total number of bucket operations. The tree underlying the Radix Sorting algorithm in which strings are assigned to leaves is a well-known data structure, a so-called {\em digital tree} or {\em trie}.  See Figure \ref{pic_radsel2} for the binary trie corresponding to the data in Figure \ref{pic_radsel}. (It is convenient to add empty leaves in such a way that all internal nodes have two children.) One observes that
the number of symbol comparisons necessary to construct the trie coincides with the complexity of Radix Sorting. Similarly, $Z_n(v)$ is equal to the number of symbol comparisons required to determine the unique leaf on the
infinite path $v$ in the trie. (Note that this leaf may or may not store a string.)

The complexity of Radix Sorting has been analyzed thoroughly in the uniform model  with precise expansions for mean and variance involving periodic functions and a central limit law for the normalized complexity, see Knuth \cite{kn98}, Jacquet and R{\'e}gnier \cite{jare88}, Kirschenhofer, Prodinger and Szpankowski \cite{kiprsz89}  and Mahmoud et al.~\cite{mafljare00}.
For the Markov source model (with $b=2$ and $0<p_{00}, p_{10} <1$) the orders of mean and variance and a central limit theorem for the complexity of Radix Sorting were derived  by Leckey, Neininger and Szpankowski \cite{lenesz15}.

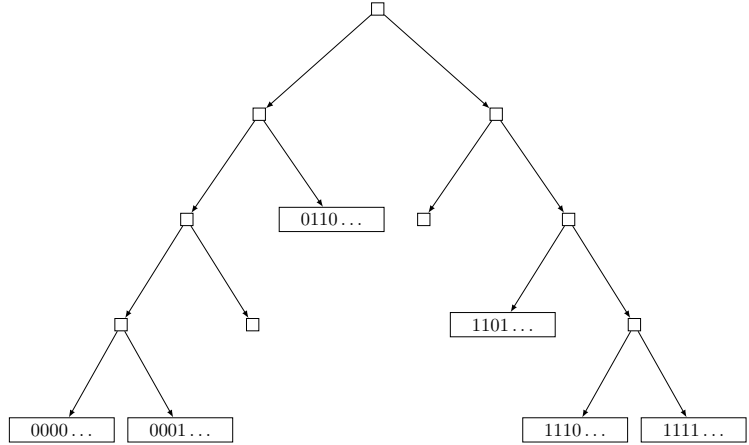
\begin{figure}
\begin{center}
\scalebox{0.7}{
\begin{tikzpicture}[edge from parent/.style={draw,-latex}]
\tikzstyle{level 0}=[level distance=1cm]
\tikzstyle{level 1}=[sibling distance=45mm,  level distance=2cm]
\tikzstyle{level 2}=[sibling distance=27.5mm,  level distance=2cm]
\tikzstyle{level 3}=[sibling distance=25mm, level distance=2cm]
\tikzstyle{level 4}=[sibling distance=22.5mm, level distance=2cm]
\node[draw]{}
  child{node[draw]{}
    child{node[draw]{}
      child{node[draw]{}
	child{node[draw,text width=5em, text centered]{$0000\ldots$}
	}
	child{node[draw,text width=5em, text centered]{$0001\ldots$}
	}
      }
      child{node[draw]{}
      }
    }
    child{node[draw,text width=5em, text centered]{$0110\ldots$}
    }
  }
  child{node[draw]{}
  child{node[draw]{}
      }
	child{node[draw]{}
      child{node[draw, text width=5em, text centered]{$1101\ldots$}}
			child{node[draw]{}
						child{node[ draw, text width=5em, text centered]{$1110\ldots$}}
						child{node[, draw, text width=5em, text centered]{$1111\ldots$}}
						}
						}
	}
;
\end{tikzpicture}
}
\end{center}
\caption{Binary trie corresponding to the data in Figure \ref{pic_radsel}.  }
\label{pic_radsel2}
\end{figure}

\medskip
{\bf Quickselect.} In a set $\mathcal S$ of $n$ data from a totally ordered set,
the Quickselect algorithm retrieves the $k$-th smallest element ($1 \leq k \leq n$) as follows: first, a pivot element $v \in \mathcal S$ is chosen and subsets $\mathcal S_< = \{w \in S : w < v\}$ and $\mathcal S_{\geq} = \mathcal S \setminus \mathcal S_<$ are generated by comparing every element to $v$. If necessary, the routine is recursively applied to the sublist which contains the sought datum. When the data are given as strings over $\Sigma$, its complexity has recently been studied under the model of symbol comparisons required to retrieve the $k$-th smallest element \cite{vaclfifl3, FillNak1, FillNak2, FillMat, clfingva}. Results in this context bear similarities with our findings in this work.

Let $C_n(k)$ denote the number of bit comparisons required to retrieve the $k$-th smallest elements in the set $\mathcal S = \{S_1, \ldots, S_n\}$ using the Quickselect routine with  uniform pivot choice (again with convention $C_n(n+1) := C_n(n)$). By Theorem 2 in Vall\'{e}e et al. \cite{vaclfifl3}, for $t \in [0,1]$, we have
\begin{align} \label{qsel} \EE{C_n(\lfloor t n \rfloor + 1)} = \varrho(t) n + O(n^{1-\delta}),\end{align}
for some $\delta > 0$.  Here,
$$\varrho(t)= 2 \sum_{v \in \Sigma^*} \pi(v)\left(1 + H\left(\frac{|t - (\mathbb P(S_1 < v) + \pi(v)/2)|}{\pi(v)} \right)\right),$$
where, with $y^+ := y + 1/2, y^- := -y + 1/2$ and $0 \log 0 := 0$,
$$H(y) = \begin{cases}  -(y^+ \log y^+ + y^- \log y^-), &\text{ if }y<1/2,\\ y^-(\log y^+ - \log |y^-|),&\text{ if } y \geq 1/2.\end{cases}$$
(Note that the definition of the term $\mu_w$ on the fourth page of \cite{vaclfifl3} appearing in the definition of $\varrho$ there needs to be replaced by $\mu_w := (1/2) (p_w^{(-)} +1 - p_w^{(+)})$ when using the notation of this paper. See \cite[Page 6, line 1]{FillNak1}.)
This result holds in the more general framework of $\Pi$-tame sources, see Corollary \ref{gensource} and the discussion thereof in Section \ref{sec:lln} for details. For an illustration of $\varrho$ and a comparison with our limit mean function $m \circ h$ with uniform initial distribution $\mu$ and $b=2$, see Figure \ref{fig10}.

The asymptotic distributional behaviour of $C_n(\lfloor t n \rfloor+1), n \geq 1$ is strikingly different from that of
$Y_n(\lfloor t n \rfloor+1), n \geq 1$. Asymptotically,
$C_n(\lfloor t n \rfloor+1), n\geq 1$
is  not concentrated around its mean and features a random almost sure limit when dividing through its expectation. See \cite[Theorem 4.1]{FillNak2}.

Finally, for a worst rank analysis corresponding to our Theorem \ref{thm:mainsup} of Quickselect in the standard comparison based model of complexity see Devroye \cite{dev01}.

\begin{figure}[h]
\begin{center}
\subfigure[$p_{00}$=0.7, $p_{10}$=0.3\label{subfig_1}]{
\mbox{
 \includegraphics[scale=0.6]{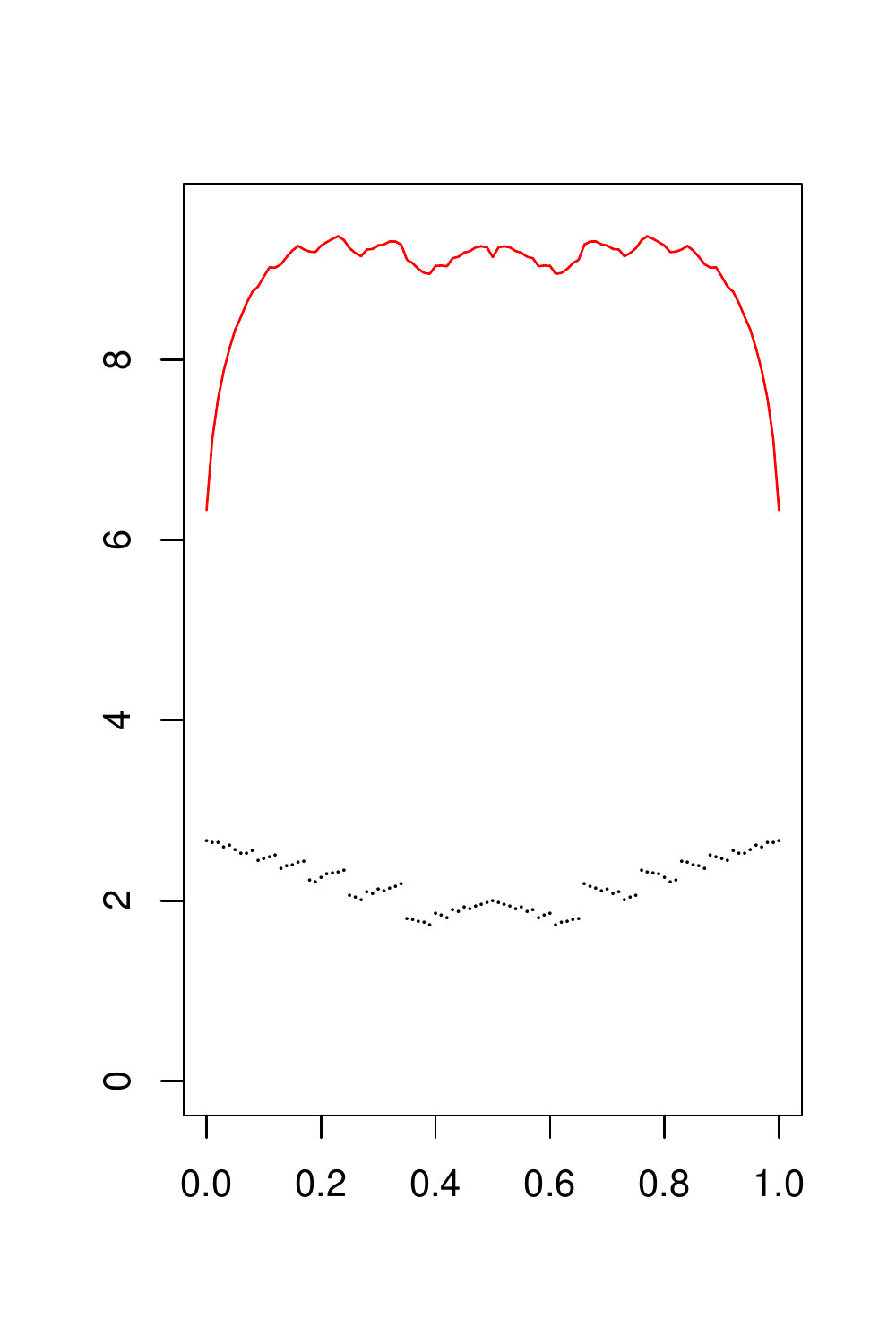}
}
}
\subfigure[$p_{00}$=0.25, $p_{10}$=0.4\label{subfig_2}]{
\mbox{
\includegraphics[scale=0.6]{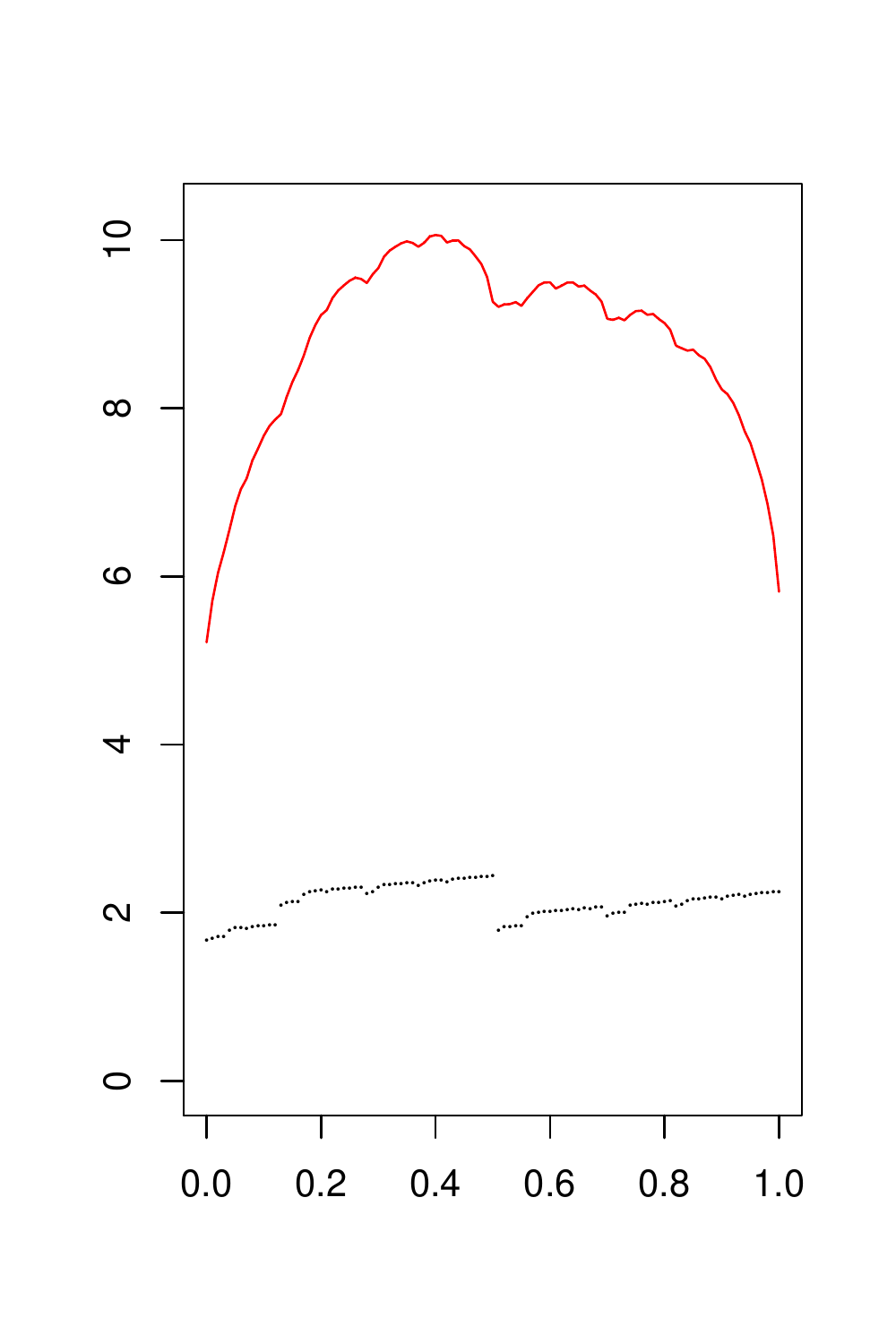}
}
}
\end{center}
\caption{Discontinuous limit mean function $m\circ h$ for Radix Selection and continuous mean function $\varrho$ for Quickselect for two sets of parameters.}
\label{fig10}
\end{figure}

\medskip
{\bf Connection to the contraction method.} The proofs of our main functional limit theorems are based on an underlying contraction argument. Thus, our approach is  strongly related to results from the contraction method, a methodology which has proved very fruitful in the probabilistic analysis of algorithms. See 
\cite{roru01, neru04} for surveys. The main idea of the contraction method is to set up a distributional recurrence for the sequence of random variables under investigation and to derive a stochastic fixed-point equation for the corresponding limiting process. Convergence is then established by recursive arguments. See the elaborate discussion preceding and preparing the proof of Theorem \ref{thm_simple2} in Section \ref{sec:uniform}. Regarding recent developments on functional limit theorems via contraction arguments, the underlying stochastic fixed-point equations characterizing the Gaussian limit processes arising in our work are functional extensions of so-called perpetuities and therefore easier to treat than applications involving Zolotarev metrics \cite{BNS, nesu15}. In particular, we are able to use techniques based on $L_2$-estimates in the spirit of the applications given in  \cite{sunedr14, raro, BS}.

\medskip
{\bf Plan of the paper.}  The paper is organized as follows: in Section \ref{sec:quantile_model} we study the quantile model starting with important definitions and preliminary results on the functions $m, F$ and $h$ in Section \ref{sec:prel}. Here, we also prove a key lemma on the asymptotic behaviour of $S_{(k)}$ when $k \sim tn$ for $t \in [0,1]$. In Section \ref{sec:lln} we prove Theorem \ref{prop:mean}, convergence of the marginal distribution in Theorem \ref{thm_simple2} and discuss extensions of these results to general probabilistic sources. Section \ref{sec:uniform} is devoted to the proof of Theorem \ref{thm_simple2} using a functional version of the contraction method. Here, we also give a characterization of $H$ by a family of distributional fixed point equations.
Section \ref{sec:second} concerns the proof of Theorem \ref{thm_simple} and a classification of transition probabilities for which the complexity admits a functional convergence result. While the proof of Theorem \ref{thm_simple} is very similar to that of Theorem \ref{thm_simple2} and not given with all details, we invest a significant amount of work on the verification of the technical steps to obtain functional convergence in the Skorokhod topology. See Proposition \ref{prop:nat}.
In Section \ref{sec:markov:av} we consider the models of grand averages and worst case rank. Here, in the short Section \ref{sec:markov:av1} we present more details on the limiting distribution from Corollary \ref{thm:maingrand}. Section \ref{sec_worst_case_rank} contains the proof of Theorem \ref{thm:mainsup} as well as further properties of the suprema of the processes. We also discuss the set $\Sigma_{\text{max}}$ there.

\medskip Some results of the present paper have been announced in the extended abstract \cite{lenesu14}.

\section{The quantile-model} \label{sec:quantile_model}

\subsection{Preliminaries} \label{sec:prel}
We start by giving more  definitions and collecting elementary properties of the functions $m, F$ and $h$.
Fix $b \geq 2$ and $\mu, (p_{ij})_{i,j \in \Sigma}$  as in the introduction.  Let
\begin{align}
\bar v &:= v_2 v_3 \ldots \mbox{ for } v=v_1v_2\ldots \in \Sigma^\infty,\label{trunc}\\
v^{(k)} &:= v_1 \ldots v_k \mbox{ for } k \geq 1,  v \in \bigcup_{n =k}^\infty \Sigma^n \cup \Sigma^\infty, v^{(0)} = \emptyset, \nonumber\\
v- &:= v_1 \ldots v_{k-1} (v_k-1) (b-1) (b-1) \ldots \in \Sigma_{b-1}^\infty \mbox{ for } v = v_1 \ldots v_k 0 0\ldots\in \Sigma^\infty_0, v_k \neq 0.\nonumber
\end{align}
Further, for $r \in \Sigma$, let $\pi_r$ ($m_r, F_r$ and $h_r$, respectively) denote the function $\pi$ ($m$, $F$ and $h$, respectively) when choosing
\begin{align} \label{rini} \mu^r := \sum_{k=0}^{b-1} p_{rk} \delta_k \end{align} as initial distribution.
We finally note that, by our assumptions on the Markov chain,
\begin{align} \label{bound:exp} r_n := \sup_{\mu} \sup_{v \in \Sigma^n} \pi(v) \leq p_{\max}^{n-1}, \quad p_{\max} := \max \{p_{ij} : i,j \in \Sigma \} < 1. \end{align}
Here the first supremum in the first expression is taken over all initial distributions of the chain.

\begin{proposition} \label{prop:elem} Let $b \geq 2$.
\begin{enumerate} 
\item For 
$v = v_1 \ldots v_k 0 0\ldots\in \Sigma_0^\infty$, we have  
\begin{equation}\label{discont_m}
\begin{aligned}
m(v)  & = \sum_{\ell = 0}^{k-1} \pi(v^{(\ell)}) + \pi(v^{(k-1)})
p_{v_{k-1}v_{k}} \left(1+ \frac{p_{v_{k}  0}}{1-p_{00}}\right), \\
m(v-)  & = \sum_{\ell = 0}^{k-1} \pi(v^{(\ell)})  + \pi(v^{(k-1)})
p_{v_{k-1}v_{k}-1} \left(1+ \frac{p_{(v_{k} -1) (b-1)}}{1-p_{(b-1)(b-1)}}\right).
\end{aligned}
\end{equation}
We also have $F(v) = F(v-)$ for all $v \in \Sigma_0^\infty$.
\item For all $v \in \Sigma^\infty$,
$$m(v) = 1 +  \mu_{v_1}  m_{v_1}(\bar v), \quad F(v) = \mu_{v_1} F_{v_1}(\bar v) + \sum_{k=0}^{v_1-1} \mu_k .$$ In particular, for all $r \in \Sigma$,
$$m_r(v) = 1 +  p_{r v_1}  m_{v_1}(\bar v), \quad F_r(v) = p_{r v_1} F_{v_1}(\bar v) + \sum_{k=0}^{v_1-1} p_{r k}.$$
 $m_0, \ldots, m_{b-1}$ are the unique bounded functions on $\Sigma^\infty$ satisfying the first system of equations in the last display. Similarly, $F_0, \ldots, F_{b-1}$ are the unique bounded functions on $\Sigma^\infty$ satisfying the second system of equations in the last display.
\end{enumerate}
\end{proposition}
\begin{proof}
 Both $i)$ and the equations in $ii)$ follow immediately from the construction. For the uniqueness claim, let $\mathfrak B(\Sigma^\infty)$ be the set of bounded functions  on $\Sigma^\infty$ endowed with the supremum norm. Equip $\mathfrak B(\Sigma^\infty)^b$ with the natural max-norm. We define the operator $T$ on this set by
 $$T(f)_r(v) = 1 +  p_{rv_1}  f_{v_1}(\bar v), \quad r \in \Sigma.$$
 Then, for $f, g \in \mathfrak B(\Sigma^\infty)^b$, one easily establishes $\| T(f) - T(g)\| \leq p_{\max} \|f-g\|$. As $p_{\max} < 1$, it follows that
 $T$ has at most one fixed point, the function $f = (m_0, \ldots, m_{b-1})$. For $F_0, \ldots, F_{b-1}$ on proceeds analogously.
\end{proof}

Theorem \ref{prop:mean} $ii)$ and Theorem \ref{thm_simple} crucially rely on the asymptotic behaviour of $S_{(k)}$ when $k/n \to t \in [0,1].$ This explains the importance of the following lemma.

\begin{lemma} \label{lem:asyS}
Let $k = k(n) \in \{1, \ldots, n\}$ with $k/n \to t \in [0,1]$.
\begin{enumerate} \item If $t \notin F(\Sigma_0^\infty)$, then $S_{(k)} \to h(t)$ almost surely. More precisely, for any $M \in \N$,
$$\sum_{n=1}^\infty \ProbB{j(S_{(k)}, h(t)) \leq M} < \infty.$$
\item If $t \in F(\Sigma_0^\infty)$, then, for any $M \in \N$,
$$\sum_{n=1}^\infty \ProbB{\max \{j(S_{(k)}, h(t)), j(S_{(k)}, h(t-))\} \leq M} < \infty.$$
\end{enumerate}

\end{lemma}
\begin{proof}
Let $v=h(t)$. We have
\begin{align*}
& \ProbB{j(S_{(k)}, v) < M}  \\ & = \ProbB{\sum_{w < v^{(M)}, w \in \Sigma^M} \Lambda_{n,M}(w) \geq k} + \ProbB{\sum_{w \leq v^{(M)}, w \in \Sigma^M} \Lambda_{n,M}(w) < k} \\
& = \ProbB{\text{Bin}\left(n, \sum_{w < v^{(M)}, w \in \Sigma^M} \pi(w)\right) \geq k} +  \ProbB{\text{Bin}\left(n, \sum_{w \leq v^{(M)}, w \in \Sigma^M} \pi(w)\right) < k}.
\end{align*}
Note that the sum in the first term is strictly smaller than $t$. Similarly, the sum in the second term is strictly larger than $t$. As $k/n \to t$, both events are large deviation events for the binomial distribution. Hence, the probabilities decay exponentially fast in $n$. This proves $i)$. The proof of $ii)$ uses similar ideas and is thus omitted.
\end{proof}

\subsection{Law of large numbers for a Markov source}  \label{sec:lln}

The following proposition is at the heart of Theorem \ref{prop:mean} $i)$ and states a weaker version of Theorem \ref{thm_simple2}.

\begin{proposition} \label{prop:weaker}
\begin{enumerate} \item Let $(v(n))_{n\geq 1}$ be an arbitrary $\Sigma^\infty$-valued sequence. Then, almost surely and with convergence of all moments,
\begin{align*}
\frac 1 {\sqrt n}\left| Z_n(v(n))-n - \sum_{i=1}^n j(v(n),S_i)\right|\to 0.
\end{align*}
\item For $k \geq 1, v_1, \ldots, v_k \in \Sigma^\infty$, with the process $H$ defined in Theorem \ref{thm_simple2}, we have
$$\left(\frac{Z_n(v_1)-m(v_1)n}{\sqrt n}, \ldots, \frac{Z_n(v_k)-m(v_k)n}{\sqrt n}\right) \stackrel{d}{\longrightarrow} (H(v_1), \ldots, H(v_k)).$$

\item For all $v,w\in\Sigma^\infty$, we have
 \begin{align*}
\E[H(v)H(w)]=\sum_{k=1}^\infty k\wedge (j(v,w)) \left(\pi(v^{(k)})+\pi(w^{(k)})\right)-\sum_{k=1}^{j(v,w)} \pi(v^{(k)})-(m(v)-1)(m(w)-1).
\end{align*}
\end{enumerate}
\end{proposition}
\begin{proof}
Note that $\sum_{k\geq 0}\Lambda_{n,k}(v)= n +  \sum_{i=1}^n j(v,S_i)$ for every $v\in\Sigma^\infty$. Consequently,
\begin{align*}
\left|Z_n(v(n)) -n - \sum_{i=1}^n j(v(n),S_i)\right| = \# \{k\geq 0 : \Lambda_{n,k}(v(n))=1\}.
\end{align*}
The right hand side is bounded by $R_n:=\max_{i\leq n} j(v(n),S_i)$. Using \eqref{bound:exp}, the union bound reveals
\begin{align*}
\Prob(R_n\geq k)\leq n r_k \leq  n p_{\max}^{k-1}.
\end{align*}
Hence $ R_n / \sqrt n  \to 0$ almost surely and with convergence of moments concluding the proof of $i)$.
$ii)$ follows from $i)$, Slutzky's lemma and the multivariate central limit theorem.
For $iii)$, set $X=j(S_1,v)$ and $Y=j(S_1,w)$. Then
 \begin{align*}
\E[XY]&=\sum_{k=1}^\infty \sum_{\ell=1}^\infty \Prob(X\geq k, Y\geq \ell)\\
&=\sum_{k=1}^{j(v,w)}\left( k\Prob(X\geq k ) +\sum_{\ell=k+1}^\infty \Prob(Y\geq \ell)\right) + \sum_{k=j(v,w)+1}^\infty j(v,w) \Prob(X\geq k)\\
&=\sum_{k=1}^\infty k\wedge (j(v,w)) \Prob(X\geq k) + \sum_{k=1}^\infty k\wedge (j(v,w)) \Prob(Y\geq k)- \sum_{k=1}^{j(v,w)} \Prob(Y\geq k)\\
&=\sum_{k=1}^\infty k\wedge (j(v,w)) \left(\pi(v^{(k)})+\pi(w^{(k)})\right)-\sum_{k=1}^{j(v,w)} \pi(v^{(k)}).
\end{align*}
Hence, since $\E[X]\E[Y]=(m(v)-1)(m(w)-1)$,
\begin{align*}
\E[H(v)H(w)]=\sum_{k=1}^\infty k\wedge (j(v,w)) \left(\pi(v^{(k)})+\pi(w^{(k)})\right)-\sum_{k=1}^{j(v,w)} \pi(v^{(k)})-(m(v)-1)(m(w)-1).
\end{align*}
This concludes the proof. \end{proof}
We can now give the proof of Theorem \ref{prop:mean} $i)$.
\begin{proof}[Proof of Theorem \ref{prop:mean} i)]
By Proposition \ref{prop:weaker} $i)$ it is sufficient to show that \begin{align*}
\frac{\sum_{i=1}^n j(v(n),S_i)} n \to m(v)-1.
\end{align*}
Since $m(v(n))\rightarrow m(v)$, the convergence above is equivalent to
\begin{align}\label{zz_coro_LLN}
\frac{\sum_{i=1}^n (j(v(n),S_i)-(m(v(n))-1))} n \to 0.
\end{align}
Note that $\E[j(v(n),S_1)]=m(v(n))-1$. Thus, the almost sure convergence follows from a suitable version of the strong law of large numbers for row-wise independent triangular arrays; cf.~\cite{HMT89}. The moment convergence follows from the almost sure convergence and the fact that the moments of the sequence in \eqref{zz_coro_LLN} are uniformly bounded.
The latter is an immediate consequence of \eqref{bound:exp}.
\end{proof}

The proof of Theorem \ref{prop:mean} $ii)$ relies on a simple tail bound for the height of the associated trie.

\begin{lemma}
Let $\mathcal H_n = \max \{k \geq 1 : \Lambda_{n,k}(v) > 1 \text{ for some } v \in \Sigma^* \}$.  Then, for all $n \geq 1$,
$$\Prob \{\mathcal H_n \geq k \} \leq n^2 p_{\max}^{k-1}, \quad k \geq 1.$$
\end{lemma}

\begin{proof}
As the bound \eqref{bound:exp} is uniform over all initial distributions, we have $$\Prob \left\{j(S_1, S_2) \geq n \right\} \leq p_{\max}^{n-1},  \quad n \geq 1.$$
The assertion follows from $\Prob \{\mathcal H_n \geq k \} \leq n^2 \Prob \left\{j(S_1, S_2) \geq k \right\}$.
\end{proof}

\begin{proof}[Proof of Theorem  \ref{prop:mean} ii)]
Let $k/n \to t \notin F(\Sigma_0^\infty)$ and $v = h(t)$.  Abbreviate $x_+ := x \mathbf 1_{x > 1}$ for $x \geq 0$. As $j(S_{(k)}, v) \to \infty$ almost surely by Lemma \ref{lem:asyS} $i)$, we need to prove a version of Theorem \ref{prop:mean} $i)$ with a random sequence $v(n)$. This relies on the concentration of the binomial distribution. To be more precise, in order to show the claimed almost sure convergence, using the Borel-Cantelli lemma, it suffices to verify that, for any $\varepsilon > 0$, we have
$$\sum_{n\geq 1} \Prob \left\{ \left| \sum_{\ell=1}^\infty n^{-1}  \Lambda_{n,\ell}(S_{(k)})_+ - \pi(v^{(\ell)}) \right| \geq \varepsilon \right\} < \infty.$$
By Lemma \ref{lem:asyS} $i)$ it is further  sufficient to prove that, for any $\varepsilon > 0$, there exists $n_0 \in \N$ such that, for all $K \geq n_0$,
$$\sum_{n\geq 1} \Prob \left\{ \left| \sum_{\ell=K}^\infty n^{-1} \Lambda_{n,\ell}(S_{(k)})_+ -   \pi(v^{(\ell)}) \right|  \geq \varepsilon, j(v,
S_{(k)}) \geq K \right\} < \infty.$$
From the tail bound in the previous lemma we deduce that it is further enough to show that, for any $\varepsilon > 0, L > 3 / \log(1 / p_{\max})$,  there exists $n_0 \in \N$ such that, for any $K \geq n_0$,
$$\sum_{n\geq 1} \Prob \left\{ \left| \sum_{\ell=K}^{L \log n} n^{-1}  \Lambda_{n,\ell}(S_{(k)})_+ -   \pi(v^{(\ell)}) \right|  \geq \varepsilon, j(v, S_{(k)}) \geq K \right\} < \infty.$$
Let $A_{K,L} \subseteq \Sigma^{L \log n}$  be the set of vectors with prefix $v^{(K)}$. Then, the last claim follows from verifying that, for any $\varepsilon > 0, L > 3 / \log(1 / p_{\max})$, there exists $n_0 \in \N$ such that, for all $K \geq n_0$,
\begin{align} \label{oneb} \sum_{n\geq 1} \Prob \left\{ \sup_{w \in A_{K,L}} \left|   \sum_{\ell=K}^{L \log n}  n^{-1} \Lambda_{n,\ell}(w)_+ -   \pi(w^{(\ell)}) \right| \geq \varepsilon  \right\} < \infty.\end{align}
For $w \in A_{K,L}$ the union bound gives
\begin{align*}
\Prob \Bigg( \Bigg|   \sum_{\ell=K}^{L \log n} n^{-1}  \Lambda_{n,\ell}(w)_+ -  & \sum_{\ell=K}^{\infty} \pi(w^{(\ell)}) \Bigg| \geq \varepsilon \Bigg) \\
& \leq L \log n \sup_{0 \leq \ell \leq L \log n} \ProbB{\left|    \Lambda_{n,\ell}(w)_+ -  n \pi(w^{(\ell)}) \right| \geq \frac{\varepsilon n}{L \log n}}.
\end{align*}
Upon choosing $K$ sufficiently large (e.g., such that $\pi(v^{(K)}) < \varepsilon /2$), standard Chernoff bounds reveal that the right hand side decays exponentially fast in $n$ uniformly in the choice of $w$. Thus, as $A_{K,L}$ contains at most $n^{L \log b}$ many vectors, the union bound concludes the proof of \eqref{oneb}. Convergence of moments follows from
Theorem \ref{thm:mainsup}. (The proof of Theorem \ref{thm:mainsup} does not rely on the statement of Theorem \ref{prop:mean}.)
\end{proof}

The next result treats the situation for values $t \in F(\Sigma^\infty_0)$.

\begin{proposition} \label{boundarycase}
Let  $k = k(n) \in \{1, \ldots, n\}$ with $k/n \to t \in F(\Sigma^\infty_0)$.
\begin{enumerate}
\item If $\sqrt{n} (k/n - t)  \to \infty$, then $n^{-1} Y_n(k) \to m \circ h(t)$ in probability.
\item If $\sqrt{n} (k/n - t)  \to -\infty$ , then $n^{-1} Y_n(k) \to m \circ h(t-)$ in probability.
\item If $\sqrt{n} |k/n - t|  \to \beta \in (-\infty, \infty)$, then $$n^{-1} Y_n(k) \stackrel{d}{\longrightarrow}   \Phi \left(\beta / (\sqrt{t(1-t)}\right)\delta_{m \circ h(t)} + \left(1- \Phi \left(\beta / (\sqrt{t(1-t)}\right)\right)\delta_{m \circ h(t-)}, $$ where $\Phi$ denotes the distribution function of a standard Gaussian random variable.
\end{enumerate}
\end{proposition}
\begin{proof} The proof proceeds along the lines of the previous one. Let $k/n \to t \in F(\Sigma_0^\infty)$ and $v = h(t)$. Note that
$v \in   \Sigma^\infty_0$. Let $\ell \geq 1$ be maximal with $v_\ell \neq 0$.
By Lemma \ref{lem:asyS}, for any fixed $K > \ell$, almost surely,
$$\mathbf 1_{S_{(k)}^{(\ell)} = v^{(\ell)}} \left| n^{-1} \sum_{r=1}^{K} \Lambda_{n,r}(S_{(k)})_+  -  \sum_{r=1}^{\ell-1} \pi(v^{(r)}) - \sum_{r=\ell}^K \pi(v^{(r)})\right| \to 0,$$
and
$$\mathbf 1_{S_{(k)}^{(\ell)} =  v^{(\ell-1)}(v_\ell-1)} \left| n^{-1} \sum_{r=1}^{K} \Lambda_{n,r}(S_{(k)})_+  -  \sum_{r=1}^{\ell-1} \pi(v^{(r)}) - \sum_{r=\ell}^K \pi(v-^{(r)})\right| \to 0.$$
As in the previous proof,  this implies
$$ \frac{Y_n(k)}{n} -  \mathbf 1_{S_{(k)}^{(\ell)} = v^{(\ell)}} m(v) - \mathbf 1_{S_{(k)}^{(\ell)} =
 v^{(\ell-1)}(v_\ell-1)} m(v-) \to 0.$$
It remains to verify that, in case $i)$, we have $\Prob \{S_{(k)}^{(\ell)} = v^{(\ell)} \} \to 1$, in case $ii)$, we have
$\Prob \{S_{(k)}^{(\ell)} = v^{(\ell)} \} \to 0$, and, in case $iii)$, $\Prob \{S_{(k)}^{(\ell)} = v^{(\ell)} \} \to \Phi(\beta / (\sqrt{t(1-t)})$.  From Lemma \ref{lem:asyS} $ii)$ it follows that there exists a non-negative sequence $\varepsilon_n, n \geq 1$ with $\varepsilon_n \to 0$, such that
\begin{align*}
\Prob \{S_{(k)}^{(\ell)} =  v^{(\ell-1)}(v_\ell-1)\} & =  \Prob \{S_{(k)}^{(\ell)} \leq  v^{(\ell-1)}(v_\ell-1)(b-1)\ldots\} - \varepsilon_n \\
& = \Prob \left \{\sum_{w \leq v-, w \in \Sigma^\ell} \Lambda_{n, \ell}(w) \geq k \right \} - \varepsilon_n\\
& = \Prob \{ \text{Bin}(n, t) \geq k \} - \varepsilon_n \\
& = \Prob \left \{ \frac{\text{Bin}(n, t) - n t}{\sqrt{n}}
\geq \frac{k - nt}{\sqrt n} \right\} - \varepsilon_n.
\end{align*}
The statements follow immediately from the central limit theorem.
\end{proof}

The last two proofs only relied on concentration bounds for the binomial distribution and the fact that $\pi(v)$ decays fast enough as the length of $v$ increases. Thus, they can easily be extended to general sources. To this end, following \cite[Definition 3]{vaclfifl3}, we call a probabilistic source $\Pi$-{\em tame} with parameter $\gamma > 0$ if, for some $C > 0$,
\begin{align*}
\max_{v \in \Sigma^k} \pi(v) = \max_{v \in \Sigma^k} \Prob\{v \preceq S_1\} \leq C k^{-\gamma}, \quad k \geq 1. \end{align*}
The arguments in the previous proofs generalize straightforwardly to $\Pi$-tame sources with parameter $\gamma > 2$.
\begin{korollar} \label{gensource}
Consider Radix Selection on a general source over the alphabet $\Sigma$.
\begin{enumerate}
\item The convergences in Theorem \ref{prop:mean} hold correspondingly almost surely and in mean, if the source is $\Pi$-tame for some $\gamma > 2$. The same is true for the multivariate central limit theorem in Proposition \ref{prop:weaker} $ii)$.
\item If the source is $\Pi$-tame for all $\gamma > 0$ then the convergences in Theorem \ref{prop:mean} are with respect to all moments.
\end{enumerate}
\end{korollar}

It is plausible that the almost sure and mean convergence hold under a weaker tameness assumption. In this context, one should point out that the mean expansion \eqref{qsel} for the Quickselect complexity holds for $\Pi$-tame sources with $\gamma > 1$ \cite[Theorem 2]{vaclfifl3}. Note however, that the multivariate central limit theorem in Proposition \ref{prop:weaker} $ii)$ requires that, for any $v \in \Sigma^\infty$, we have $\pi(v^{(n)}) = O(n^{-2})$  since, otherwise, the variance of $H(v)$ is infinite.

\subsection{Functional limit theorem for a Markov source} \label{sec:uniform}
For a refined asymptotic analysis we normalize the process in space and consider
$X_n=(X_n(v))_{v \in \Sigma^\infty}$ defined  by
\begin{align}\label{def:proc}
 X_n(v):= \frac{Z_n(v)-m(v) n}{\sqrt{n}}, \quad n \geq 1.
\end{align}
Note that we have already proved convergence of the marginal distributions of this process in Proposition \ref{prop:weaker} $ii)$.
We write $Z_n^{(r)}, r \in \Sigma$ ($X_n^{(r)}, r \in \Sigma$ respectively) for the process $Z_n$ ($X_n$ respectively) when the initial distribution is chosen as $\mu^{r}$ defined in \eqref{rini}. We now outline the ideas of the proof of Theorem \ref{thm_simple2}. 

\medskip
\noindent
{\bf Outline of the analysis:}
To set up recurrences for the processes $Z_n, n \geq 1$ and $X_n, n \geq 1$ we let  $I^{n}=(I^{n}_0,\ldots,I^{n}_{b-1})$ denote the numbers of elements in the $b$ buckets after distribution of all $n$ elements in the first partitioning stage. We obtain, recalling notation (\ref{trunc}),  \begin{align}\label{rec_rn1}
Z_n \stackrel{d}{=} \left( Z^{(v_1)}_{I^{n}_{v_1}}\left(\bar v \right)+n\right)_{v \in \Sigma^\infty},
\end{align}
where $(Z^{(0)}_{j}),\ldots,(Z^{(b-1)}_{j}),I^{n}$ are independent.

For the normalized processes $X_n$ in (\ref{def:proc}), using Proposition \ref{prop:elem} $iii)$,  we obtain 
\begin{align}
X_n &\stackrel{d}{=} \Bigg( \sqrt{\frac{I^{n}_{v_1}}{n}}X^{(v_1)}_{I^{n}_{v_1}}(\bar v) +
 m_{v_1}(\bar v) \frac{I^{n}_{v_1}- \mu_{v_1} n}{\sqrt{n}}
\Bigg)_{v \in \Sigma^\infty}, \label{rec_mod}
\end{align}
with conditions on independence and distributions as in (\ref{rec_rn1}).

From the underlying probabilistic model it follows that the vector $I^{n}$ has the multinomial $M(n;\mu_1,\ldots,\mu_{b-1})$ distribution. Hence, we have
$\frac{1}{n}I^{n}\to (\mu_1,\ldots,\mu_{b-1})$ almost surely as $n\to\infty$ and
\begin{align} \label{limN}
\frac{I^{n}-(\mu_1,\ldots,\mu_{b-1})n}{\sqrt{n}} \stackrel{d}{\longrightarrow} N,
\end{align}
where $N$ has the multivariate normal distribution ${\cal N}(0,\Omega)$ with mean zero and covariance matrix~$\Omega$ given by
$\Omega_{ij} =\mu_i(1-\mu_i)$ if $i=j$ and $\Omega_{ij} =-\mu_i \mu_j$ if $i\neq j$.
Note that $\sum_{r=0}^{b-1}  N_r  = 0$ almost surely. Similarly, we write $N^{(r)}, r \in \Sigma$ for a Gaussian random variable on $\mathbb R^b$ with zero mean and covariance matrix $\Omega^{(r)}$ given by
$\Omega^{(r)}_{ij} = p_{ri}(1-p_{ri})$ if $i=j$ and $\Omega_{ij} =-p_{ri} p_{rj}$ if $i\neq j$.

Let $\mathfrak B : \mathbb R^b \to \Co$ denote the bounded linear operator
$$\mathfrak B(v)(w) =   m_{w_1}(\bar w) v_{w_1}, \quad w \in \Sigma^\infty.$$
Motivated by \eqref{rec_mod} and \eqref{limN}, we associate the limit equation
\begin{align}\label{fpe_rn}
X\stackrel{d}{=} \Bigg( \sqrt{\mu_{v_1}}  X_{v_1}(\bar v) +\mathfrak{B}(N)(v)\Bigg)_{v \in \Sigma^\infty},
\end{align}
where $X_0,\ldots,X_{b-1}, N$ are independent, and $X_0, \ldots, X_{b-1}$ satisfy the system of fixed-point equations
\begin{align}\label{fpe_rn_sys}
X_r  \stackrel{d}{=} \Bigg( \sqrt{p_{rv_1}}  X_{v_1}(\bar v) +\mathfrak{B}(N^{(r)})(v)\Bigg)_{v \in \Sigma^\infty} , \quad r \in \Sigma,
\end{align}
with conditions on independence as in the previous line.
Considering recurrence \eqref{rec_mod} and limiting equation \eqref{fpe_rn}, it suffices to prove the functional limit theorem for the processes $X_n^{(r)}, n \geq 0, r \in \Sigma$.

\smallskip For $r \in \Sigma$,  let $H_r$ denote the Gaussian process defined in Theorem \ref{thm_simple2} when the initial distribution of the Markov source is $\mu^r$ given in 
\eqref{rini}.
The contraction arguments in the proof below show that
$H_0, \ldots, H_{b-1}$ is the unique set of random variables (in distribution) satisfying \eqref{fpe_rn_sys} with values in $\Co$ under the condition $\sup_r \EE{\|H_r\|^3} < \infty$. In fact, we have the following stronger result.
\begin{proposition} \label{prop:uniq}
The family $H_0, \ldots, H_{b-1}$  is the unique family of random variables (in distribution) in $\Co$ satisfying \eqref{fpe_rn_sys}.
\end{proposition}

The remaining part of this section is devoted to the proofs of Theorem \ref{thm_simple2} and Proposition \ref{prop:uniq}.

\begin{proof}[Proof of Theorem \ref{thm_simple2}]
The main proof idea is to construct versions of the sequences $X^{(r)}_n, n \geq 0, r \in \Sigma$ and limits $H_0, \ldots, H_{b-1}$ on the same probability space in such a way, that the statement of the theorem holds with convergence in probability. The approach uses Skorokhod's representation theorem. It has proved fruitful in a number of similar problems, e.g.\ in \cite{GR, sunedr14}. In the process, one also constructs  solutions to \eqref{fpe_rn_sys} on an almost sure level.

\medskip \noindent \textbf{A family of tries.} In a trie generated by pairwise distinct infinite strings $s_1, \ldots, s_n$, for $v \in \Sigma^*$, let $J(v) = \#\{ 1 \leq k \leq n: v \preceq s_k \}$.  If $J(v) \geq 2$ and $J(vr) \geq 1$ for $v \in \Sigma^*, r \in \Sigma$, then $J(vr)$ denotes the number of strings stored in the subtree rooted at $vr$. Therefore, we call $J(v)$ the \emph{nontruncated} subtree size of node $v$. The family $\{J(v), v \in \Sigma^*\}$ determines the trie, that is, it allows to reconstruct  the order statistics $s_{(1)}, \ldots, s_{(n)}$.  Below, we use this observation to define random tries through their nontruncated subtree sizes. Note that, in our trie constructed from $S_1, \ldots, S_n$, the nontruncated subtree sizes evolve as follows: $J(\emptyset) = n$ and $(J(0), \ldots, J(b-1))$ has the distribution of $I^{(n)}$. Then, conditionally on $J(v), v \in \cup_{\ell \leq k} \Sigma_\ell$ for $k \geq 1$,
upon writing $w^+$ for the last symbol of a vector $w$, we have:
\begin{itemize}
\item the random variables $(J(v0), \ldots, J(v(b-1)), v \in  \Sigma_{k+1}$ are independent, and
\item for $v \in  \Sigma_{k+1}$, the vector $(J(v0), \ldots, J(v(b-1))$ has the multinomial $(J(v); p_{v^+0}, \ldots, p_{v^+(b-1)})$ distribution.
\end{itemize}

Let $\{ (I^{(r), n, v })_{n\geq 0}, N^{(r), v} : v \in \Sigma^*\}$, $r \in \Sigma$, be independent families of independent and identically distributed random variables where
 $I^{(r), n, v}$ has the multinomial $M(n;p_{r0},\ldots,p_{r(b-1)})$ distribution, and $N^{(r), v}$ has the distribution of $N^{(r)}$ with \begin{align} \label{startingpoint}
\frac{I^{(r), n, v}-(p_{r0},\ldots,p_{r(b-1)}) n }{\sqrt{n}}\rightarrow N^{(r),v}, \quad n \to \infty,
\end{align}
almost surely for all $v \in \Sigma^*, r \in \Sigma$. Such a family exists by Skorokhod's representation theorem.
For any $v \in \Sigma^*, r,k \in \Sigma,  n \geq 0$, set $J_n^{(r), v}(\emptyset) = n$ and $J_n^{(r), v}(k) = I_k^{(r), n, v}$.
Then, recursively, for $v \in \Sigma^*, r \in \Sigma,  n \geq 0$ we define
$$J_n^{(r), v}(wk) = I_k^{(w^+),J_n^{(r), v}(w) , vw}, \quad w \in \bigcup_{\ell = 1}^\infty \Sigma^\ell, k \in \Sigma.$$
By construction,
$\{J_n^{(r), v}(w) : w \in \Sigma^*\}$ is distributed like $\{J(w) : w \in \Sigma^*\}$ defined above when the initial distribution of the chain is $\mu^r$ given in \eqref{rini}.
Hence, it is the family of nontruncated subtree sizes of a trie generated by $n$ infinite strings with order statistics   $V^{(r),n,v}_{1} < \cdots < V^{(r),n,v}_{n}$ where  $(V^{(r),n,v}_{1}, \ldots, V^{(r),n, v}_{n})$ is distributed like $(S_{(1)}, \ldots, S_{(n)})$.  Observe however that these tries do not almost surely grow, that is, with $\mathcal V^{(r), v}_n := \{V^{(r),n,v}_{1}, \ldots, V^{(r),n,v}_{n}\}$, we do not almost surely have $ \mathcal V^{(r), v}_n \subseteq \mathcal V^{(r), v}_{n+1} $.
 The $\Co$-valued random variables $Z_n^{(r), v}, v \in \Sigma^\infty$ are now defined as in \eqref{def:z} but based on $\mathcal V^{(r), v}_n$. By construction, for any $ v \in \Sigma^*, r \in \Sigma, n \geq 0$, we have $Z_n^{(r),v} \stackrel{d}{=} Z_n^{(r)}$ and
$$Z_n^{(r), v}(w)  = Z^{(w_1), vw_1}_{I_{w_1}^{(r), n, v}}(\bar w) + n, \quad w \in \Sigma^\infty.$$

\medskip \noindent \textbf{The limit process.} Setting $H_0^{(r), v} := 0$ for all $v \in \Sigma^*, r \in \Sigma$, we recursively construct random variables $H_n^{(r), v},  v \in \Sigma^*, r \in \Sigma, n \geq 1,$ by
$$H_{n+1}^{(r), v}(w) = \sqrt{p_{rw_1}}H_n^{(w_1), vw_1}(\bar w) +\mathfrak{B}(N^{(r), v})(w), \quad w \in \Sigma^\infty. $$
It follows that, for any $p \in \N$,
\begin{align}
 \sup_r  \EE{\|H_{n+1}^{(r), v} - H_{n}^{(r), v}\|^p} & \leq \sup_r  \sum_{k=0}^{b-1} p_{rk}^{p/2} \EE{\|H_{n}^{(k), v k} - H_{n-1}^{(k), vk}\|^p} \nonumber \\
& \leq \sup_r \sum_{k=0}^{b-1} p_{rk}^{p/2}   \sup_r \E[ \| H^{(r), v}_{n} - H^{(r), v}_{n-1}\|^p]  \nonumber \\
& \leq \left( \sup_r \sum_{k=0}^{b-1} p_{rk}^{p/2}\right)^n \sup_r \E[ \| H^{(r), v}_{1} \|^p]. \label{boundp}
\end{align}
From here, choosing $p > 2$, standard arguments (see, e.g.\ the proof of Lemma 2.1 in \cite{sunedr14}) show that, almost surely, $H_n^{(r),v}, n \geq 0$ is uniformly Cauchy for all $v \in \Sigma^*, r \in \Sigma$. By the completeness of $\Co$, the processes are uniformly convergent. As intended, the limits denoted by $H_r^{v}$ satisfy
$$H_r^{v}(w) = \sqrt{p_{rw_1}} H^{vw_1}_{w_1}(\bar w) +\mathfrak{B}(N^{(r), v})(w), \quad w \in \Sigma^\infty.$$
In particular, $H_0^{v}, \ldots, H_{b-1}^v$ satisfy \eqref{fpe_rn}. From \eqref{boundp}, it follows that $\E[\|H_n^{(r),v} - H_r^{v}\|^p] \to 0$ for all
$p \in \N$. By construction, we also have the following useful series representation: almost surely, for any $v \in \Sigma^*, r \in \Sigma$,
\begin{align} \label{inf_sum_rep} H_r^v(w) = \sum_{s=0}^{\infty} \sqrt{\pi_r(w^{(s)})} m_{w_{s+1}}(w_{s+2}w_{s+3} \ldots) N^{(w_s), vw^{(s)}}_{w_{s+1}}, \quad w \in \Sigma^\infty.
\end{align}

\medskip \noindent \textbf{Convergence of the discrete process.} For $v \in \Sigma^\infty, r \in \Sigma, n \geq 0,$  set
$$X_n^{(r), v}(w) = \frac{Z_n^{(r), v}(w) - m_r(w)n}{\sqrt n}, \quad w \in \Sigma^\infty.$$ By construction, we have $X_n^{(r)} \stackrel{d}{=} X_n^{(r), v}$.
In the context of the contraction method, it has turned out fruitful to define an accompanying sequences by replacing the coefficients in the system of limiting equations  by the corresponding terms in the distributional recurrence.
For all $v \in \Sigma^*, r \in \Sigma, n \geq  0$, let
\begin{align*}
Q_n^{(r), v}(w)  & =   \sqrt{\frac{I^{(r), n,v}_{w_1}}{n}}H_{w_1}^{vw_1}(\bar w)+ \mathfrak B (N^{(r), v})(w), \quad w \in \Sigma^\infty.
\end{align*}
By construction, the random variables $H^{v0}, \ldots, H^{v(b-1)}, (I^{(r), n,v}, N^{(r), v})$ are independent and their joint  distribution does not depend on $v$. Hence, the same follows for $Q_n^{(r), v}$ and $Q_n^{(r), v} - H_r^v$. Further,  by induction over the length of the vector $v$, it is straightforward to verify that
the distribution of $X_n^{(r),v } - Q_n^{(r), v}$ does not depend on $v$. We omit these details.

The proof of $\Delta_n := \sup_r \EE{ \| X_n^{(r),v} - H^{v}_r \|^3} \to 0$ is standard in the context of the contraction method.
First of all, we have
\begin{align*} \sup_r \EE{ \| Q_n^{(r),v} - H^v_r \|^3} & \leq \sup_r \E\left[ \sum_{k=0}^{b-1} \left | \sqrt{\frac{I^{(k), n,v}_r}{n}} - \sqrt{p_{rk}} \right |^3 \left \| H_k^{vk} \right \|^3 \right]  \\
& \leq b \cdot \sup_r \E\left[\sup_k \left | \sqrt{\frac{I^{(r), n,v}_k}{n}} - \sqrt{p_{rk}} \right |^3\right] \sup_r \E \left[ \left \| H_r \right \|^3\right]. \end{align*}
 The right hand side tends to zero by \eqref{startingpoint}. (We only use the law of large numbers here.)
By the triangle inequality, it follows that
\begin{align*}
\Delta_n^{1/3}  \leq &  \left(\sup_r \EE{ \| X_n^{(r),v} - Q_n^{(r), v} \|^3}\right)^{1/3} + o(1) \\
& \leq \left(\sup_r \sum_{k=0}^{b-1} \EE{ \left( \frac{I^{(r), n,v}_k}{n} \right)^{3/2} \left \|  X^{(k), vk}_{I_k^{(r), n, v}} - H^{vk}_k \right\|^3}\right)^{1/3} \\
 & + \sup_r \sup_k \left(\EE{\left|\frac{I^{(r), n,v}_k- p_{rk}n}{\sqrt {n}} - N^{(r), v}_k \right|^3}\right)^{1/3} + o(1).
\end{align*}
By \eqref{startingpoint}, the second summand on the right hand side turns to zero as $n \to \infty$.
It follows that
\begin{align*}
\Delta_n^{1/3} \leq \left(\sup_r \sum_{k=0}^{b-1} \EE{ \left( \frac{I^{(r), n,v}_k}{n} \right)^{3/2} \Delta_{I^{(r), n,v}_k}}\right)^{1/3} + o(1).
\end{align*}
As \begin{align} \label{cont_con} \EE{\left(\frac{I^{(r),n, v}_k}{n}\right)^{3/2}} \to p_{rk}^{3/2}, \quad \sup_r \sum_{k=0}^{b-1} p_{rk}^{3/2} < 1, \end{align} a simple induction on $n$ shows that $\Delta_n$ is bounded. In a second step, by the contraction argument used in \eqref{boundp}, one can show that $\Delta_n \to 0$. We omit the details which are standard in the framework of the contraction method and refer the reader to
the proof of, e.g.\ \cite[Proposition 3.3]{sunedr14} or \cite[Theorem 4.1]{neru04}  where these arguments are worked out in detail.

Finally, the covariance function of $H$ in the uniform model can easily be computed using the stochastic fixed-point equation \eqref{fpe_rn}, as $X, X_1, \ldots, X_{b-1}$ are identical in distribution.
\end{proof}
\begin{proof}[Proof of Proposition \ref{prop:uniq}]
Assume that $X_0, \ldots, X_{b-1}$ (more precisely, the corresponding distributions) satisfy \eqref{fpe_rn_sys}. Our aim is to prove that $X_r$ has the distribution of $H_r$.
We first show that, for all $r \in \Sigma, v \in \Sigma^\infty,$ the random variable $X_r(v)$ has a mean zero Gaussian distribution. To this end, we expand the system \eqref{fpe_rn_sys} for several levels. To be more precise, using this system of equations, by induction (over $\ell$) it is straightforward to show that, for all $v \in \Sigma^\infty, r \in \Sigma$ and $\ell \geq 0$, setting $v_0 = r$, we have
\begin{align} \label{dep} X_r(v) \stackrel{d}{=} \sqrt{\pi_r(v^{(\ell)})} X_{v_\ell} (v_{\ell+1} v_{\ell+2} \ldots) + \sum_{s=0}^{\ell-1} \sqrt{\pi_r(v^{(s)})} m_{v_{s+1}}(v_{s+2}v_{s+3} \ldots) N^{(v_s)}_{v_{s+1}},\end{align}
with independent random variables $N^{(v_0)}, \ldots, N^{(v_\ell)}, X_{v_\ell}$ where the distributions of $N^{(v_0)}, \ldots, N^{(v_\ell)}$ were introduced in the previous proof. In particular, choosing
$r = 0$ and $v = 00\ldots$, we have
$$X_0(v) \stackrel{d}{=} \sqrt{\pi_0(v^{(\ell)})} X_{0} (v) + \sum_{s=0}^{\ell-1} \sqrt{\pi_0(v^{(s)})} m_{0}(0) N^{(0)}_0,$$
Classical results from the theory of perpetuities, see, e.g.\ Theorem 1.5 in \cite{Vervaat}, show that this identity uniquely determines the distribution of $X_0(v)$. Thus, $X_0(v)$ has a mean zero normal distribution. From \eqref{dep} it follows that, for all $v \in \Sigma_0^\infty$, $X_r(v)$ has a mean zero Gaussian distribution. By continuity, the same holds for all $v \in \Sigma^\infty$.

Next, we need to verify that, for all $q \geq 2, \lambda_1, \ldots, \lambda_q \in \mathbb R, w_1, \ldots, w_q \in \Sigma^\infty$ and $r \in \Sigma$ the random variable
$\lambda_1 X_r(w_1) + \ldots \lambda_q X_r(w_q)$ has a mean zero Gaussian distribution. For the sake of presentation, we consider the case $q=2$ and write
$\theta = w_1, \sigma = w_2$. Assume that $\theta \neq \sigma$ and set $j := j(\theta, \sigma)$  defined in \eqref{def:coin}.  Our aim is to expand the system \eqref{fpe_rn_sys} on a functional level. More precisely, we define coupled versions of $X_0, \ldots, X_{b-1}$ by backward induction as follows:
first, let $\{\mathcal X_r^v : v \in \Sigma^{j+1}\}$, $r \in \Sigma$  be independent families of  independent copies of $X_r$. We also assume these families to be independent of the random variables introduced in the previous section. Then, recursively, for
$v \in \Sigma^k$, $0 \leq k \leq j$, $r \in \Sigma$, set
$$\mathcal X_r^v(w) =  \sqrt{p_{rw_1}} \mathcal X_{w_1}^{vw_1}(\bar w) + \mathfrak{B}(N^{(r), v})(w), \quad w \in \Sigma^\infty. $$
Since $X_0, \ldots, X_{b-1}$ satisfy \eqref{fpe_rn_sys}, the same is true for $\mathcal X_0^v, \ldots, \mathcal X_{b-1}^v$ for all $v\in \Sigma^k, 0 \leq k \leq j+1$.
Similarly to \eqref{dep}, we have
\begin{align} \label{the}\mathcal X_r^\emptyset(\theta) = \sqrt{\pi_r(\theta^{(j+1)})} \mathcal X_{\theta_{j+1}}^{\theta^{(j+1)}} (\theta_{j+2} \theta_{j+3} \ldots)
+ \sum_{s=0}^{j} \sqrt{\pi_r(\theta^{(s)})} m_{\theta_{s+1}}(\theta_{s+2}\theta_{s+3} \ldots) N^{(\theta_s), \theta^{(s)}}_{\theta_{s+1}}, \end{align}
and
\begin{align} \label{sig} \mathcal X_r^\emptyset(\sigma) = \sqrt{\pi_r(\sigma^{(j+1)})} \mathcal X_{\sigma_{j+1}}^{\sigma^{(j+1)}} (\sigma_{j+2} \sigma_{j+3} \ldots)
+ \sum_{s=0}^{j} \sqrt{\pi_r(\theta^{(s)})} m_{\sigma_{s+1}}(\sigma_{s+2}\sigma_{s+3} \ldots) N^{(\theta_s), \theta^{(s)}}_{\sigma_{s+1}}. \end{align}
In particular,
\begin{align*}
 \lambda_1  X_r(\theta) & + \lambda_2 X_r(\sigma)  \\
 \stackrel{d}{=} & \lambda_1 \sqrt{\pi_r(\theta^{(j+1)})} X_{\theta_{j+1}} (\theta_{j+2} \theta_{j+3} \ldots) + \lambda_2 \sqrt{\pi_r(\sigma^{(j+1)})} X_{\sigma_{j+1}}(\sigma_{j+2} \sigma_{j+3} \ldots) \\
& + \sum_{s=0}^{j-1}  (\lambda_1 m_{\theta_{s+1}}(\theta_{s+2}\theta_{s+3} \ldots) + \lambda_2 m_{\theta_{s+1}}(\sigma_{s+2}\sigma_{s+3} \ldots))
\sqrt{\pi_r(\theta^{(s)})} N^{(\theta_s)}_{\theta_{s+1}} \\
& + \lambda_1  m_{\theta_{j+1}}(\theta_{j+2}\theta_{j+3} \ldots) N^{(\theta_j)}_{\theta_{j+1}} + \lambda_2
m_{\sigma_{j+1}}(\sigma_{j+2}\sigma_{j+3} \ldots) N^{(\theta_j)}_{\sigma_{j+1}},
\end{align*}
with independent random variables $N^{(\theta_0)}, \ldots, N^{(\theta_j)}, X_{\theta_{j+1}}, X_{\sigma_{j+1}}$.
By the first part of the proof the right hand side has a zero mean Gaussian distribution.
Further, \eqref{the} and \eqref{sig} determine $\EE{X_r(\theta) X_r(\sigma)}$. This concludes the proof.
\end{proof}

\subsection{Proof of Theorem \ref{thm_simple}} \label{sec:second}

In this section, we discuss the asymptotic behaviour of the sequence
$$\mathcal Y_n(t) := \frac{Y_n(\lfloor t n \rfloor +1) - m\circ h(t) n}{\sqrt{n}}, \quad t \in [0,1].$$
Our first result is the natural extension of Theorem \ref{thm_simple2} to the process $Z_n(S_{(\lfloor t n \rfloor +1)}), n \geq 1$.
\begin{proposition} \label{prop:nat}
In distribution, with respect to the Skorokhod topology on $\Do$, we have
$$\left(\frac{Z_n(S_{(\lfloor t n \rfloor +1)}) - m(S_{(\lfloor t n \rfloor +1)})n}{\sqrt{n}}\right)_{t \in [0,1]} \to H\circ h.$$
\end{proposition}
\begin{proof}
Recall $X_n$ from \eqref{def:proc}.
By Theorem \ref{thm_simple2} we have $X_n \circ h \to H \circ h$ in $\Do$ in distribution. By Slutzky's lemma, it remains to show that, for some metric $d$ generating the Skorokhod topology on $\Do$, we have $d(X_n (S_{(\lfloor \cdot n \rfloor +1)}), X_n \circ h) \to 0$ in probability. As the sequence of distributions of $X_n, n \geq 1,$ is tight, it is enough to find a family of strictly increasing continuous bijections $\lambda_n, n \geq 1$ on $[0,1]$ such that, in probability,
\begin{align} \label{conv_dsk} \inf_{t \in [0,1]} j(S_{(\lfloor \lambda_n(t) n \rfloor +1)}, h(t))  \to \infty, \quad \text{and } \| \lambda_n - \text{id} \| \to 0.\end{align}
In fact, we show that both convergences hold in the almost sure sense. Let
$$h_n = \max \{ \ell \geq 0 : \Lambda_{n,\ell}(v) > 1 \text{ for all } v \in \Sigma^\ell \}.$$
$h_n$ is the lowest level of a node in the associated trie with outdegree strictly smaller than $b$. $h_n$ is monotonically increasing and $h_n \to \infty$ almost surely.
By construction, for  any $v \in \Sigma^{h_n}$, we can choose $t_v \in \{0, 1/n, 2/n, \ldots, (n-1)/n\}$ minimal with $v \preceq S_{(t_v n +1)}$. Note that both $t_v < t_w$ and $F(v00\ldots) < F(w00\ldots)$ if and only if $v < w$. Further, we have $t_{00\ldots0} = 0$ and $t_{(b-1)(b-1)\ldots(b-1)} \leq (n-1)/n$. Let us now argue that, for $\ell \geq 1, v \in \Sigma^\ell$, almost surely,
\begin{align}\label{eq33} t_{v00\ldots 0} \to F(v00\ldots), \quad n \to \infty.\end{align}
Assume for a contradiction that $t_{v00\ldots0} \geq F(v00\ldots) + \varepsilon$ for some $\varepsilon > 0$ and infinitely many $n$. Then, for those values of $n$,
$S_{(\lfloor t_{v00\ldots0} n\rfloor + 1)} \geq S_{(\lfloor (F(v00\ldots) + \varepsilon)n\rfloor +1)}$. By Lemma \ref{lem:asyS} $ii)$ the term on the right hand side takes values strictly larger and bounded away from $v00\ldots$ for all $n$ sufficiently large. This contradicts the fact that $S_{(\lfloor t_{v00\ldots0} n\rfloor + 1)} \to v00\ldots$ almost surely following from the definition of $t_{v00\ldots0}$. An analogous argument applies to the case $t_{v00\ldots0} \leq F(v00\ldots) - \varepsilon$ for infinitely many $n$. Summarizing, we have verified \eqref{eq33}. Now, we set
$$\lambda_n(F(v00\ldots)) = t_v, \quad v \in \Sigma^{h_n}, \quad \lambda_n(1) = 1.$$
Upon linearly interpolating $\lambda_n$ between successive values $F(v00\ldots), F(w00\ldots), v < w \in \Sigma^{h_n},$ and $t = 1$,
the function $\lambda_n$ is a strictly increasing piecewise linear bijection on the unit interval. By construction, $j(S_{(\lfloor \lambda_n(t) n \rfloor +1)}, h(t)) \geq h_n$ proving the first statement in \eqref{conv_dsk}.

Using the piecewise linearity of
$\lambda_n$, we further deduce
\begin{align*}\|\lambda_n(t) - t  \| = \sup_{v \in \Sigma^{h_n}} |F(v00\ldots) - t_v |.
\end{align*}
Fix a large constant $K \in \N$. In the remainder of the proof assume that $n$ is sufficiently large such that $h_n > K$. Then, the right hand side of the last display is bounded from above by
\begin{align} \label{3sums} \sup_{v \in \Sigma^{h_n}} |F(v00\ldots) - F(v^{(K)}00\ldots) | + \sup_{v \in \Sigma^{h_n}} |t_v - t_{v^{(K)}00\ldots0} |+  \sup_{v \in \Sigma^K} |F(v00\ldots) - t_{v00\ldots0}  |.\end{align}
We now show that each of these three terms can be made arbitrarily small for all $n$ sufficiently large upon choosing $K$ large enough. For the first summand, this follows immediately from the uniform continuity of $F$. For the third summand, it follows  from \eqref{eq33}. In order to analyze the second term in \eqref{3sums}, for $v \in \Sigma^K, v \neq (b-1)(b-1)\ldots(b-1)$ let $v^+ \in \Sigma^K$ denote the smallest vector in $\Sigma^K$ strictly larger than $v$. Further, let
$t_{(b-1)(b-1)\ldots(b-1)^+00\ldots0} := 1$. Then, the second summand is bounded by
$$ \sup_{v \in \Sigma^{K}} |t_{v^+00\ldots0} - t_{v00\ldots0} | \to \sup_{v \in \Sigma^{K}} |F(v^+00\ldots) - F(v00\ldots)|, $$
where the limit is taken as $n \to \infty$. Upon choosing $K$ sufficiently large, the right hand can be made arbitrarily small by uniform continuity of $F$. \end{proof}

\begin{proposition} \label{prop:D2}
If, for some (then all) $r \in \Sigma$, there exists $v \in \Sigma_0^\infty$ with $m_r(v) \neq m_r(v-)$, then, for all $0 \leq a < b \leq 1$, there exists $a \leq t \leq b$ such that the sequence of distributions of $\mathcal Y_n(t)$ is not tight. In particular, the sequence of distributions of $\mathcal Y_n$ considered on $\Doab$ is not tight.
\end{proposition}

\begin{proof}
From the recursiveness of the model it follows that, if $m_r(v) \neq m_r(v-)$ for some $r \in \Sigma, v \in \Sigma_0^\infty$, then, there exists a set
$A \subseteq \Sigma_0^\infty$ which is dense in $\Sigma^\infty$ (with respect to the topology $\mathcal T_\infty$) such that $m(v) \neq m(v-)$ for all $v \in A$. Accordingly, $B := \{F(v), v \in A\}$ is dense in $[0,1]$ and both $h(t-) \neq h(t)$ and $m\circ h(t-) \neq m \circ h(t)$  for all $t \in B$. Fix $v \in A$ and $t_0 = F(v)$. By Lemma \ref{lem:asyS}, we have
$\max\{j(S_{(\lfloor t_0n \rfloor +1)}, v), j(S_{(\lfloor t_0n \rfloor +1)}, v-) \} \to \infty$ almost surely.  Further, by the argument in the proof of Proposition \ref{boundarycase}
relying on the central limit theorem for the binomial distribution, we have $\ProbB{S_{(\lfloor t_0n \rfloor +1)} \leq v-} \to 1/2$.
With
\begin{align} \label{vn} v_n(t) := \left(m(S_{(\lfloor t n \rfloor +1)}) - m\circ h(t)\right) \sqrt{n}, \quad t \in [0,1], \end{align}
it follows that,  for any
$\varepsilon > 0$ sufficiently small, we have $$\liminf_{n \to \infty} \ProbB{|v_n(t_0)| > \varepsilon \sqrt {n}} \geq 1/2.$$
Clearly, the sequence of distributions of $v_n(t_0), n \geq 1$ is not tight. The same follows for the sequence of distributions of $\mathcal Y_n(t_0)$ from Proposition \ref{prop:nat}. \end{proof}

\medskip \textbf{Remark.}
Proposition \ref{prop:D2} holds analogously upon replacing the function $m \circ h$ in the definition of $\mathcal Y_n$ by either
$$\left(\frac{1}{2}(m\circ h(t) + m \circ h(t-))\right)_{t \in [0,1]}, \quad \text{or} \quad \left(\EE{m(S_{(\lfloor t n \rfloor +1)})}\right)_{t \in [0,1]}.$$

It turns out that continuity and linearity are equivalent for the function $m \circ h$. (This statement is not necessarily true for a Markov source with $p_{ij} = 0$ for some $i,j \in \Sigma$.)
\begin{proposition} \label{prop:elem2} Let $b \geq 2$.
\begin{enumerate}
\item We have $m_r \circ h_r \in \Co$ for some (then all) $r \in \Sigma$ if and only if $m_r(v) = m_r(v-)$ for all $v \in \Sigma_0^\infty$ and some (then all) $r \in \Sigma$. This is the case if and only if, for some $\beta > 0$, we have \begin{align} \label{condd} p_r := p_{0r} = \ldots = p_{(b-1) r} = \frac{1-\beta}{1-\beta^b} \beta^r, \quad \text{for all }r \in \Sigma. \end{align}
(For $\beta = 1$, this means $p_r = 1/b$ for all $r \in \Sigma$.) Unless this condition is satisfied, for all $r \in \Sigma$, the function $m_r \circ h_r$ is discontinuous on a subset of $F_r(\Sigma_0^\infty)$ that is dense in $[0,1]$.
\item For a Markov source with weights chosen as in \eqref{condd}, we have
$$m_0(v) = \ldots = m_{b-1}(v) = \alpha F(v)  + \gamma, $$ where $\alpha, \gamma$ are defined in \eqref{def:ga}. In particular, $m_r \circ h_r(t) = \alpha t + \gamma, t \in [0,1]$.
\item In the uniform model, $m, F$ and $h$ are given by
$$m(v) = \frac{b}{b-1}, \quad F(v) = \sum_{k\geq 1} v_k b^{-k}, \quad t= \sum_{k \geq 1} h(t)_k b^{-k}.$$
In the Bernoulli model (with $b=2$), we have
\begin{align} \label{m_bernoulli} m(v) = \frac{1-2p_0}{p_0(1-p_0)} F(v)  + \frac 1 {1-p_0}. \end{align}
\end{enumerate}
\end{proposition}

\begin{korollar} \label{cor:notight}
Consider Radix Selection with a general Markov source not satisfying condition \eqref{condd} for any $\beta > 0$. Then, for all $0 \leq a < b \leq 1$, the sequence of distributions of $\mathcal Y_n, n \geq 1$ considered on
$\Doab$ is not tight.
\end{korollar}

We leave it as an open problem to decide whether the one- (or finite-) dimensional marginal distributions of the process $(\mathcal Y_n)_{n \geq 1}$ converge weakly for a general Markov source with transition probabilities different from \eqref{condd}.

\begin{proof} [Proof of Proposition \ref{prop:elem2}]
 First of all, if $m_r \circ h_r$ has a point of discontinuity on $(0,1)$ for some $r \in \Sigma$, then, by Proposition \ref{prop:elem} $iii)$,  this discontinuity reproduces in all functions $m_i \circ h_i, i \in \Sigma$ on all scales. Next, we show that $m_r(v) = m_r(v-)$ for all $r \in \Sigma, v \in \Sigma_0^\infty$ implies that the source is memoryless. To this end, using \eqref{discont_m} once with $v_{k-1} = i, v_k=j$  and then with $v_{k-1} = \ell, v_k = j$ where
 $i, \ell \in \Sigma$ and $j = 1, \ldots, b-1$ shows that $p_{ij} / p_{\ell j} = p_{i(j-1)} /p_{\ell (j-1)}$. By induction over $j$ we obtain
 $$\frac{p_{ij}}{p_{\ell j}} =  \frac{p_{im}}{p_{\ell m}}, \quad \text{or } p_{ij} p_{\ell m} = p_{\ell j} p_{im}  \quad \text{for all } i,j,\ell, m \in \Sigma.$$
Summation over $m$ reveals $p_{ij} = p_{\ell j}$. As $i, j, \ell$ are arbitrary, the source is memoryless.
The condition $m_r(v) = m_r(v-)$ for all $r \in \Sigma, v \in \Sigma_0^\infty$ and \eqref{discont_m} imply that, for a memoryless source and $r =1, \ldots, b-1$, we have $$p_r = p_{r-1} \frac{1-p_0}{1-p_{b-1}} = p_0 \left( \frac{1-p_0}{1-p_{b-1}}\right)^r.$$ It is now straightforward to verify that \eqref{condd} yields the set of all possible solutions. This concludes the proof of $i)$.
 To show $ii)$, as $m = m_0 = \ldots =m_{b-1}$, we only need to prove that the function $\alpha F_0(v) + \gamma$ solves the equation stated in Proposition \ref{prop:elem} $iii)$. This boils down to a routine calculation using the corresponding identity for $F$ also stated in Proposition  \ref{prop:elem} $iii)$, that is,
 $F(v) = p_0 + \cdots + p_{v_1 - 1} + p_{v_1} F(\bar v)$. $iii)$ follows from $ii)$ upon computing the relevant values of $\alpha$ and $\gamma$.
\end{proof}

The remainder of this section is devoted to the proof of Theorem \ref{thm_simple}. Subsequently, we assume that \eqref{goodcase} holds. Then, recalling
\eqref{vn}, we have
$$\mathcal Y_n(t) = \frac{Z_n(S_{(\lfloor t n \rfloor +1)}) - m(S_{(\lfloor t n \rfloor +1)})n}{\sqrt{n}} + v_n(t).$$
For $\beta = 1$ we have $v_n = 0$, and the statement follows immediately from Proposition \ref{prop:nat}.

For $\beta \neq 1$, we start as for the process $Z_n$ with a recursive decomposition. By the memorylessness of the source, we may drop the index $r$.
To simplify notation, we set
\begin{align}\label{convention}
[a,1):= [a,1] \quad \mbox{for} \quad 0<a<1.
 \end{align}
Setting $F^n_{-1} := 0$ and $F^n_r = \sum_{j=0}^r I^n_j$, $r \in \Sigma$, we obtain (under convention \eqref{convention})
\begin{align*}
\mathcal Y_n &\stackrel{d}{=} \left( \sum_{r=0}^{b-1} {\bf 1}_{\big [\frac{F^{n}_{r-1}}{n}, \frac{F^{n}_{r}}{n}\big )}(t) \left(\sqrt{\frac{I^{n}_r}{n}}
\mathcal Y^{(r)}_{I^{n}_r}
\left(\frac{nt-F^n_{r-1}}{I_r^{n}}\right) +  \frac{\gamma I^{n}_r - \alpha F_{r-1}^n - (\gamma-1) n}{\sqrt{n}}
\right)\right)_{0\le t\le 1}, 
\end{align*}
with conditions on independence and distributions as in (\ref{rec_rn1}). (As opposed to the situation in \eqref{rec_rn1}, the sequences $(\mathcal Y^{(0)}_n), \ldots, (\mathcal Y^{(b-1)}_n)$ are now identically distributed.)

To associate to recurrence  (\ref{rec_mod}) a limit equation in the spirit of the contraction method, we introduce  a family of parameter transformations: for $0 \leq a < b \leq 1$, under convention \eqref{convention}, let
\begin{align*}
\mathfrak{A}_{a,b}: \Do \to \Do, \quad \mathfrak{A}_{a,b}(f)(t) = {\bf 1}_{[a,b)} f \left( \frac{t-a}{b-a}\right).
\end{align*}
Then we associate the limit equation (again with convention (\ref{convention}))
\begin{align}\label{fpe_rn_G}
G\stackrel{d}{=}\sum_{r=0}^{b-1} \sqrt{p_r} \mathfrak{A}_{\sum_{k=0}^{r-1} p_k, \sum_{k=0}^r p_k} (G_r)+ {\bf 1}_{\left[\left.\sum_{k=0}^{r-1} p_k, \sum_{k=0}^r p_k\right)\right.}(\cdot) \left( \gamma N_r - \alpha \sum_{k=0}^{r-1} N_k \right),
\end{align}
where $G_0,\ldots, G_{b-1}, N$ are independent, and $G_0, \ldots, G_{b-1}$ are distributed like $G$ while \sloppy $N = (N_0, \ldots, N_{b-1})$ is distributed as in \eqref{limN}. (Note that the additive term on the right hand side has mean zero for all $r \in \Sigma.$) The covariance function can be computed recursively:  first, let $b = 2$.
For $s\in [0,p_0), t\in [p_0,1]$, \eqref{fpe_rn_G} implies that
\begin{align*}
\EE{G(s)G(t)}=\EE{\gamma N_0 (\gamma N_1 -\alpha N_0)}
\end{align*}
since $G_0$, $G_1$ and $(N_0, N_1)$ are independent. $N_0 = - N_1$ and the definition of $\alpha,\gamma$ yield $\EE{G(s)G(t)} =-1$.
Now consider the case $(s,t)\in [0,p_0)^2 \cup [p_0,1)^2$. For $x\in[0,1]$ let
\begin{align*}
\bar x =\begin{cases}  x/ {p_0},&\text{ if }x<p_0,\\ (x-p_0)/p_1,&\text{ otherwise.}\end{cases}
\end{align*}
Finally, let $w_1={\bf 1}_{[p_0,1]}(t)$. Then, \eqref{fpe_rn_G} implies
\begin{align*}
\EE{G(s)G(t)}&=p_{w_1} \EE{G(\bar s) G(\bar t)}+\V\left(\gamma N_{w_1}-\alpha \sum_{k=1}^{w_1-1} N_k\right)\\
&=p_{w_1} \EE{G(\bar s) G(\bar t)}+\frac{p_{w_1}}{1-p_{w_1}}.
\end{align*}
Iterating this formula  yields \eqref{CV_G_pf}. The same strategy can be worked out in the general case $b\geq 2$, but explicit formulas become rather heavy. Therefore, we only state the result for the variance: for $t \in [0,1]$, we have
\begin{align} \label{exp_var}
\text{Var}(G(t)) = \E[G(t)^2]&=\frac{1-\beta^{b+1}}{\beta-\beta^b} (\alpha t + \gamma) - \frac{\beta(1-\beta^b)^2}{(\beta-\beta^b)^2}.
\end{align}
A formal proof of the functional convergence $\mathcal Y_n \to G$ requires slight modifications of the arguments in the proof of Theorem \ref{thm_simple2} and the aligning of discontinuity points worked out in the proof of Proposition \ref{prop:nat}. Most of these steps are not novel at this point, and we remain brief.

\begin{proof}[Proof of Theorem \ref{thm_simple}]
We work in the setting of the proof of Theorem \ref{thm_simple2}. We assume $\beta \neq 1$, that is, $\alpha \neq 0$. As in the construction of the limit process $H$, let $\tilde H_0^v = 0$ for all $v \in \Sigma^*$ and, recursively, for $v \in \Sigma^*, n \geq 0,$
$$\tilde H_{n+1}^v(w) = \sqrt{p_{w_1}} \tilde H_n^{vw_1}(\bar w) + \gamma N^v_{w_1} - \alpha \sum_{k=0}^{w_1-1} N_k^v.$$
As in \eqref{boundp}, we obtain a bound of the form
\begin{align} \label{htilde} \EE{\|\tilde H_{n+1}^v - \tilde H_n^v \|^3} \leq C q^n, \quad C > 0, 0 < q < 1.\end{align}
In particular, there exists a family of $\Co$-valued random variable $\{\tilde H^v : v \in \Sigma^*\}$ such that, almost surely, $\|\tilde H_n^v - \tilde H^v\| \to 0$ and
$$\tilde H^v(w) = \sqrt{p_{w_1}} \tilde H^{v w_1} (\bar w) + \gamma N^v_{w_1} - \alpha \sum_{k=0}^{w_1-1} N_k^v, \quad w \in \Sigma^\infty.$$
Analogously to \eqref{inf_sum_rep}, we have
\begin{align*}  \tilde H^v(w) = \sum_{s=0}^{\infty} \sqrt{\pi(w^{(s)})}
\left(\gamma N^{(w_s), vw^{(s)}}_{w_{s+1}} - \alpha \sum_{k=0}^{w_{s+1}-1} N^{(w_s), vw^{(s)}}_k\right), \quad w \in \Sigma^\infty.
\end{align*}
Clearly, $\tilde H^v(00\ldots) = H^v(00 \ldots)$ and $\tilde H^v(11\ldots) = H^v(11 \ldots)$, but, for $w \notin \{00\ldots, 11\ldots\}$, almost surely, $\tilde H^v(w) \neq H^v(w)$. We set $G^v = \tilde H^v \circ h$. Next,
we choose $Y_n^v(k) = Z_n^v(V^{k,v}_{n}), 1 \leq k \leq n$, to obtain versions of the process $Y_n$.
By construction, for the corresponding processes $\mathcal Y_n^v(t), t \in [0,1]$, upon defining $F^{n,v}$ in the obvious way, we obtain (again under convention \eqref{convention})
$$\mathcal Y_n^v(t) = \sum_{r = 0}^{b-1} \sqrt{\frac{I_r^{n,v}}{n}} \mathfrak{A}_{\frac{F^{n,v}_{r-1}}{n}, \frac{F^{n,v}_{r}}{n}} \left(\mathcal Y^{vr}_{I_r^{n,v}}\right) +
 \sum_{r = 0}^{b-1} {\bf 1}_{\big [\frac{F^{n,v}_{r-1}}{n}, \frac{F^{n,v}_{r}}{n}\big )}(t) \frac{\gamma I^{n,v}_r - \alpha F_{r-1}^{n,v} - (\gamma-1) n}{\sqrt{n}}, \quad t \in [0,1].$$
In order to connect $\mathcal Y_n^v$ to $G^v$, define an accompanying sequence recursively as follows: let $R_0^v = 0$ for all $v \in \Sigma^*$, and, for $n \geq 1$, with convention \eqref{convention},
$$R_n^v(t) = \sum_{r = 0}^{b-1} \sqrt{p_r} \mathfrak{A}_{\frac{F^{n,v}_{r-1}}{n}, \frac{F^{n,v}_{r}}{n}} \left( R^{vr}_{I_r^{n,v}}\right) +
 \sum_{r = 0}^{b-1} {\bf 1}_{\big [\frac{F^{n,v}_{r-1}}{n}, \frac{F^{n,v}_{r}}{n}\big )}(t) \left(\gamma N^v_{r} - \alpha \sum_{k=0}^{r-1} N_k^v\right), \quad t \in [0,1].$$
The convergence $\EE{\|\mathcal Y_n^v - R_n^v\|^3} \to 0$ uses the contraction argument underlying the proof of Theorem \ref{thm_simple2} based on
\eqref{startingpoint} and \eqref{cont_con}. We omit to repeat these steps. By defining $\lambda^v_n$ ($h_n^v$, respectively) in
the trie constructed from $\mathcal V^v_{ n}$ as $\lambda_n$ ($h_n$, respectively)  in the proof of Proposition \ref{prop:nat}, we can guarantee \eqref{conv_dsk}, that is, in probability, $\| \lambda_n^v - \text{id} \| \to 0$. Further, we have, almost surely,
$$\|R_n^v \circ \lambda^v_n - G^v \| \leq 2 \sum_{k= h_n^v}^\infty \|\tilde H^v_{k+1} - \tilde H^v_k\|.$$ Upon recalling \eqref{htilde}, the right hand side tends to zero in probability as $h^v_n \to \infty$ in probability. This concludes the proof.
\end{proof}

\section{Grand averages and worst case rank}\label{sec:markov:av}

\subsection{The model of grand averages}\label{sec:markov:av1}

We now consider  the complexity of Radix Selection with $b\geq2$ buckets assuming the Markov source model for the data and the model of grand averages for the rank. \begin{proposition} \label{propZ}
Let $M$ have the  multinomial $(1; \mu_0, \ldots, \mu_{b-1})$ distribution and $\xi$ be uniformly distributed on $[0,1]$.
The distribution of $Z := m \circ h(\xi)$ is given by
\begin{align}\label{nn3101}
Z\stackrel{d}{=} 1 + \sum_{k=0}^{b-1} M_k \mu_k Z_k,
\end{align}
where $M,Z_0, \ldots, Z_{b-1}$ are independent, and the distributions of $Z_0,  \ldots, Z_{b-1}$ are the unique solutions of the following system:
\begin{align*} 
Z_r &\stackrel{d}{=} 1 + \sum_{k=0}^{b-1} M^{(r)}_{k} p_{rk} Z_k, \quad r \in \Sigma.
\end{align*}
Here, $M^{(0)}, \ldots, M^{(b-1)}, Z_0, \ldots, Z_{b-1}$ are independent and, for $r \in \Sigma$,  $M^{(r)}$ has the multinomial $(1; p_{r0}, \ldots, p_{r(b-1)})$ distribution. Further,
$$\EE{Z} = \sum_{v \in \Sigma^*} \pi(v)^2 = 1 + \sum_{k=0}^{b-1} \mu_k^2 \EE{Z_k}.$$
For $b = 2$, we have
\begin{align*}
\EE{Z_0} & = \frac{1+p_{01}^2-p_{11}^2}{2(p_{00}+p_{11})(1+p_{00}p_{11})-2(p_{00}+p_{11})^2}, \\
\EE{Z_1} & = \frac{1+p_{10}^2-p_{00}^2}{2(p_{00}+p_{11})(1+p_{00}p_{11})-2(p_{00}+p_{11})^2}.
\end{align*}
Similarly,
$$\EE{Z^2} = \sum_{v \in \Sigma^*} \pi(v)^2 \left( \pi(v) + 2  \sum_{w \preceq v, w \neq v} \pi(w)\right)  = 2 \EE{Z}-1 + \sum_{k=0}^{b-1} \mu_k^3 \EE{Z^2_k}.$$

\end{proposition}

\begin{proof}
\eqref{nn3101} is a direct consequence of Proposition \ref{prop:elem} $iii$).
Iterating the system \eqref{nn3101} shows that $\Law(Z_r)$ satisfies a one-dimensional fixed-point equation.
Let $A_r = \{ v  \in \Sigma^* : v_1, \ldots, v_{k-1} \in \Sigma \setminus \{r\}, v_k = r, k \geq 1\}$. Now, let $(A,B)= (\pi_r(V), \sum_{w \preceq V, w \neq V} \pi_r(w))$ where $\Prob\{V = v\} = \pi_r(v)$, $v \in A_r$.
 Then $\Law(Z_r)$ satisfies $Z_r \stackrel{d}{=} A Z_r + B$ where $Z_r, (A,B)$ are independent. It is well-known that fixed-point equations of this type have unique solutions (in distribution) under very mild conditions \cite[Theorem 1.5]{Vervaat}.  The formulas for expectations and second moments follows immediately from the system of fixed-point equations.
\end{proof}

In principle, the system of fixed-point equations allows to obtain explicit expressions for higher moments of the limiting distributions. However, precise formulas are lengthy and provide little insight.

\medskip \noindent \textbf{Remark:} For $b=2$, in the \emph{anti-symmetric} case, that is, $p := p_{00} = p_{11}$, a symmetry argument shows that $\Law(Z_0) = \Law(Z_1)$, and that this distribution is characterized by the fixed-point equation
$$ Z_0 \stackrel{d} = (B_{p} p + (1-B_p)(1-p)) Z_0 +1,$$
where $B_p$ has the Bernoulli distribution with success probability $p$ and $B_p, Z_0$ are independent. This is the same fixed-point equation as in the symmetric Bernoulli case $p:= p_{00} = p_{10}$. From \eqref{m_bernoulli} we know that, in distribution, $$Z_0 = Z_1 = \frac{1-2p}{p(1-p)} \xi + \frac{1}{1-p}.$$ Note that this is consistent with Figures \ref{subfig_b} and \ref{subfig_d} on page \pageref{subfig_a}. In both figures, the (closure of the) images of both red and blue functions are equal to the interval $[1/\max(p_{00}, p_{01}), 1/\min(p_{00}, p_{01})]$.
In the general case, the limiting distributions are harder to describe.   By classical results going back to Grincevi{\v{c}}jus \cite{Grin}, it is well-known that, under very mild conditions, perpetuities  such as $\Law(Z_0)$ and $\Law(Z_1)$ for $b=2$ are either absolutely continuous, singularly continuous or discrete.
It is easy to see that both laws are non-atomic, and we leave a more elaborate discussion of their properties for future work.

\subsection{The worst case rank model}\label{sec_worst_case_rank}

We now discuss the worst case rank model. Theorem \ref{thm:mainsup} follows easily from Theorem  \ref{thm_simple2}.

\begin{proof} [Proof of Theorem \ref{thm:mainsup}] As in Section \ref{sec:uniform}, set $X_n(v) = (Z_n(v) - m(v)n) / \sqrt{n}$. By Theorem \ref{thm_simple2} and Skorokhod's representation theorem, we may assume that $\| X_n - H \| \to 0$ almost surely.  For $M_n = \sup_{v \in \Sigma^\infty} Z_n(v)$, we have

$$\liminf_{n \to \infty} \frac{ M_n - m_{\text{max}} n}{\sqrt{n}} \geq \sup_{v \in \Sigma_{\text{max}}^\infty} H(v).$$
Thus, we need to show that
\begin{align} \label{suppart} \limsup_{n \to \infty} \frac{ M_n - m_{\text{max}} n}{\sqrt{n}} \leq \sup_{v \in \Sigma_{\text{max}}^\infty} H(v).\end{align}
Let $\varepsilon > 0$. By uniform continuity of $H$ and uniform convergence of $X_n, n \geq 0$, there exist (random) $M, N_0 \in \N$ such that
$$|X_n(v) - H(w)| \leq \varepsilon \quad  \text{for all } n \geq N_0, j(v,w) \geq M.$$
Further, by uniform continuity of $m$, there exists $\varepsilon_1 > 0$ such that
$$m(v) \geq m_{\text{max}} - \varepsilon_1 \Rightarrow j(v, w) \geq M \text{ for some } w \in \Sigma_{\text{max}}^\infty.$$
In the remainder of the proof assume $n \geq N_0$. By construction, on the one hand,
$$\frac{Z_n(v) - m_{\text{max}}n}{\sqrt{n}} \leq \sup_{v \in \Sigma_{\text{max}}^\infty} H(v) + \varepsilon \quad \text{for all } v \in \Sigma^\infty  \text{ with } j(v,w) \geq M \text{ for some } w \in \Sigma_{\text{max}}^\infty.$$
On the other hand, if $j(v, w) < M$ for all $w \in \Sigma_{\text{max}}^\infty$, then $m(v) \leq m_{\text{max}} - \varepsilon_1$ and therefore
$$\frac{Z_n(v) - m_{\text{max}}n}{\sqrt{n}} \leq  \sup_{v \in \Sigma^\infty} H(v) + \varepsilon - \varepsilon_1 \sqrt n.$$
Hence, $$\limsup_{n \to \infty} \frac{ M_n - m_{\text{max}} n}{\sqrt{n}} \leq \sup_{v \in \Sigma_{\text{max}}^\infty} H(v) + \varepsilon.$$
As $\varepsilon$ was chosen arbitrarily, we obtain \eqref{suppart} concluding the proof of the distributional convergence.

For the convergence of the moments note that the proof of $\E[\| X_n^{(r),v} -H_r^v\|^3] \rightarrow 0$  in the verification of Theorem \ref{thm_simple2} can easily be extended to show that, for any $p\geq 3$,
$$\E[\| X_n^{(r),v} -H_r^v\|^p]  \rightarrow 0.$$
Since
$$\inf_{v \in \Sigma^\infty} X_n(v) \leq   \frac{M_n- m_{\text{max}}n}{\sqrt{n}} \leq \sup_{v \in \Sigma^\infty} X_n(v),$$
this concludes the proof.
\end{proof}

\medskip \textbf{Remark.} Our proof of the distributional convergence in Theorem \ref{thm:mainsup} extends straightforwardly to any sequence of random variables with values in the space of continuous functions on an arbitrary compact metric space $K$ satisfying a functional convergence as in Theorem \ref{thm_simple2}.

\medskip \noindent  In the context of centered continuous Gaussian processes, it is well known that boundedness of the variance function leads to bounds on the variance and the tails of the supremum. The following results follow directly from, e.g., Theorem 5.8 in \cite{boluma13}.
\begin{proposition}\label{suptails}
With $H$ in Theorem \ref{thm_simple2} let $\sigma_{\text{max}}^2 = \max_{v \in \Sigma^\infty} \EE{H(v)^2}$.
For the supremum $S=\sup_{v \in \Sigma^\infty}H(v)$ and $t>0$, we have
$$\Prob(|S-\E[S]|\geq t)\leq 2\exp\left(-\frac{t^2}{2\sigma_{\text{max}}^2} \right).$$
Moreover, $$\V(S) \leq \sigma_{\text{max}}^2.$$
For a memoryless source, we have
$$\sigma_{\text{max}}^2 = \sup_{v \in \Sigma^\infty} m(v)(1-m(v)) + 2 \sum_{k=1}^\infty k \pi(v^{(k)}).$$ In the uniform model, $\sigma_{\text{max}}^2 = \frac{b}{(b-1)^2}$.

The analogous bounds apply to the process $G$ in Theorem \ref{thm_simple}.
\end{proposition}
Finally, let us discuss the structure of the set $\Sigma_{\text{max}}$ in the context of some examples. For $r \in \Sigma$, we write $\Sigma^r_{\text{max}}$ for $\Sigma_{\text{max}}$ if the initial distribution is $\mu^{r}$ defined in \eqref{rini}.

\medskip \textbf{Example I: The unique case.} For almost all choices of transition probabilities, the set $\Sigma_{\text{max}}$ contains exactly one element. The situation in Theorem \ref{thm_simple} (with $\beta \neq 1$) yields just one possible example.

\medskip \textbf{Example II: The finite case.} Let $b=2$. It is easy to construct a source with $\Sigma_{\text{max}} = \{00\ldots, 11\ldots \}$ setting
$p_{00} = p_{11}>1/2$ and $\mu_0 = \mu_1 = 1/2$. The situation is more complicated for the set $\Sigma_{\max}^0$ since, for $p_{00} = p_{11}$, this set is
not finite. Let $b = 4$. Then, it should be clear that, if we choose $p_{11}, p_{23}, p_{31}$ very close to 1  and $p_{02},
p_{03}$ very close to $1/2$, then only the strings $2311\ldots$ and $311\ldots$ can lie in $\Sigma_{\max}^0$. A straightforward calculation shows that this set contains both strings if we choose $p_{11} = p_{31} = 1-\varepsilon$, $p_{02} = 1/2 + \varepsilon, p_{03} = 1/2 - 2 \varepsilon, p_{23} = 1 - \varepsilon(2\varepsilon+7)/(2 \varepsilon+1)$ and $\varepsilon$ sufficiently small.

\medskip \textbf{Example III: The countable case.} Let $b = 4$ and $p_{00} = p_{01} = p_{11} = 1/3$, $p_{02} = p_{03} = 1/6$, $p_{10} = p_{12} = p_{13} = 2/9$. Further, let $p_{2r} = p_{3r} = 1/4$ for all $r \in \Sigma$. Then, $$\Sigma_{\max}^0 = \{v11\ldots : v = 00\ldots0 \in \Sigma^*\} \cup \{00\ldots\}.$$
$\Sigma_{\max}^0$ and $F_0(\Sigma^0_{\text{max}})$  are countably infinite.

\medskip \textbf{Example IV: A set of Cantor type.} Let $b=3$ and $p_{r0} = p_{r2} = 2/5$, $p_{r1} = 1/5$ for $r =0,2$ and $p_{1k} = 1/3$ for $k \in \Sigma$.
Then $$\Sigma_{\max}^0 = \{v \in \Sigma^\infty: v_i \in \{0,2\} \text{ for all } i \geq 1 \}.$$
$F_0(\Sigma^0_{\text{max}})$ is a perfect set with Hausdorff dimension $\log 2 / \log (5/2) = 0.756\ldots$ (See, e.g.\ \cite[Example 4.5]{falconer}.)

\begin{center} \textbf{Acknowledgements} \end{center}
The research of the second author was supported by DFG grant NE 828/2-1.
The research of the third author was supported by the FSMP, reference: ANR-10-LABX-0098, and a Feodor Lynen Research Fellowship of the Alexander von Humboldt
Foundation.

\section*{Appendix}

Algorithm \ref{Alg:RadSel} describes Radix Select on strings. Here, we assume numbers are given in their $b$-ary expansions over the alphabet $\{0,\ldots,b-1\}$ and let
\begin{itemize}
\item $k$ denote the sought rank, 
\item A = $[s_1,\ldots,s_n]$ be the input list of size $n$ with strings $s_1, \ldots, s_n$,
\item length($B$) denote the number of strings in a list of strings $B$,
\item $B[j]$ denote the $j$-th string in a list of strings $B$, and
\item $s_j[\ell]$ denote the $\ell$-th symbol of the string $s_\ell$.
\end{itemize}

\begin{algorithm}[h]
\caption{Radix Select}\label{Alg:RadSel}
\begin{algorithmic}[]
\Procedure{RadSel}{{\bf int} $k$, {\bf int} $b$, {\bf list} $A$}
\Return \Call{RSLoop}{$k$,$b$,$A$,$1$}
\EndProcedure

\Procedure{RSLoop}{{\bf int} $k$, {\bf int} $b$, {\bf list} $A$, {\bf int} $x$}
\Comment{Iteration in \Call{RadSel}{}}

\If{$k>\text{length}(A)$}
\Return $\text{'invalid input'}$

\EndIf

\If{$\text{length}(A) =1$}
\Return $A[1]$

\Else

\For{$i \in \{0,\ldots,b-1\}$}\Comment{Initializing the 'buckets'}

\State $A_i\gets\text{empty list}$

\EndFor

\For{$s\in A$}\Comment{Distributing data into 'buckets'}

\State $A_{s[x]}\gets A_{s[x]}\cup \{ s\}$

\EndFor

\State $L\gets 1$
\State $F\gets 0$

\While{$F+\text{length}(A_L)<k$}\Comment{Finding $L$ with $s_{(k)}\in A_L$}

\State $F\gets F +\text{length}(A_L)$
\State $L\gets L+1$

\EndWhile

\Return \Call{RSLoop}{$k-F$, $b$, $A_L$, $x+1$}\Comment{Continue Search in $A_L$}
\EndIf

\EndProcedure
\end{algorithmic}
\end{algorithm}

\end{document}